\documentclass[12pt]{amsart}
\usepackage{amssymb,amsmath, amsthm,latexsym,verbatim,stmaryrd}
\usepackage{a4wide}
\usepackage[hypertex]{hyperref}\usepackage{enumitem}

\theoremstyle{plain}
\newtheorem{theorem}{Theorem}
\newtheorem{lem}{Lemma}[section]
\newtheorem{theo}[lem]{Theorem}
\newtheorem{prop}[lem]{Proposition}
\newtheorem{corollary}[lem]{Corollary}

\theoremstyle{definition}

\theoremstyle{remark}
\newtheorem {bmrk}[theorem] {Remark}

\newcommand{\C}{\mathbb{C}}
\newcommand{\R}{\mathbb{R}}

\newcommand{\N}{\mathbb{N}}
\newcommand{\bH}{\mathbb{H}}
\newcommand{\cH}{\mathcal{H}}

\newcommand{\gL}{\mathfrak{g}}
\newcommand{\kL}{\mathfrak{k}}
\newcommand{\mL}{\mathfrak{m}}
\newcommand{\aL}{\mathfrak{a}}
\newcommand{\pL}{\mathfrak{p}}
\newcommand{\nf}{\mathfrak{n}}
\newcommand{\hL}{\mathfrak{h}}
\newcommand{\bL}{\mathfrak{b}}

\newcommand{\Spin}{\operatorname{Spin}}

\newcommand{\SO}{\operatorname{SO}}
\newcommand{\Tr}{\operatorname{Tr}}
\newcommand{\tr}{\operatorname{tr}}
\newcommand{\Id}{\operatorname{Id}}
\newcommand{\Ad}{\operatorname{Ad}}
\newcommand{\vol}{\operatorname{vol}}
\newcommand{\End}{\operatorname{End}}
\newcommand{\GL}{\operatorname{GL}}

\newcommand{\bs}{\backslash}

\begin{document}

\title[]
{The asymptotics of the Ray-Singer analytic torsion for compact 
hyperbolic manifolds}
\date{\today}

\author{Werner M\"uller}
\address{Universit\"at Bonn\\
Mathematisches Institut\\
Endenicher Allee 60\\
D -- 53115 Bonn, Germany}
\email{mueller@math.uni-bonn.de}

\author{Jonathan Pfaff}
\address{Universit\"at Bonn\\
Mathematisches Institut\\
Endenicher Alle 60\\
D -- 53115 Bonn, Germany}
\email{pfaff@math.uni-bonn.de}

\keywords{analytic torsion, hyperbolic manifolds}
\subjclass{Primary: 58J52, Secondary: 11M36}

\begin{abstract}
In this paper we study the asymptotic behavior of the analytic torsion for
compact oriented hyperbolic manifolds with respect to certain rays of 
irreducible representations. 
\end{abstract}

\maketitle
\setcounter{tocdepth}{1}
\section{Introduction}
\setcounter{equation}{0}

Let $G=\Spin(d,1)$ and $K=\Spin(d)$. Then $K$ is a maximal compact subgroup of 
$G$ and $\widetilde{X}:=G/K$ can be identified with the hyperbolic space of 
dimension $d$. Let $\Gamma\subset G$ be a discrete, torsion free co-compact
subgroup. Then $X=\Gamma\bs\tilde X$ is a compact oriented hyperbolic 
manifold of dimension $d$ and every such manifold is of this form. Let $\rho$
be a finite-dimensional representation of $\Gamma$ on a complex vector space 
$V_\rho$. Let $E_\rho\to X$ be the associated flat vector bundle. Pick a 
Hermitian
fibre metric $h$ in $E_\rho$. Let $\Delta_p(\rho)$ denote the Laplacian on
$E_\rho$-valued $p$-forms on $X$. Let $\zeta_p(s;\rho)$
 be the zeta function of $\Delta_p(\rho)$ (see \cite{Shubin}). It is a
meromorphic function of $s\in\C$, which is holomorphic at $s=0$. Then the
Ray-Singer analytic torsion $T_X(\rho;h)\in\R^+$ is defined by
\begin{equation}\label{anator1}
\log T_X(\rho;h):=\frac{1}{2}\sum_{p=1}^d(-1)^p p \frac{d}{ds}\zeta_p(s;\rho)
\big|_{s=0}
\end{equation}
(see \cite{RS}, \cite{Mu1}). In general, $T_X(\rho;h)$ depends on $h$. If
$\rho$ is unitary and $\dim X$ is even, then $T_X(\rho)=1$ (see 
\cite[Theorem 2.3]{RS}). Furthermore, if $H^*(X,E_\rho)=0$ and $\dim  X$ is 
odd, then by \cite[Corollary 2.7]{Mu1}, $T_X(\rho;h)$ is independent of $h$.

In this paper we consider the special class of representations of $\Gamma$
which are obtained as restriction to $\Gamma$ of 
finite-dimensional irreducible representations of $G$. 
Let $\tau$ be a finite-dimensional irreducible representation of $G$ and
let $E_\tau$
denote  the flat bundle over $X$ associated to $\tau|_\Gamma$. 
By \cite[Lemma 3.1]{Mats},
$E_\tau$ can be equipped with a distinguished metric, called admissible, which 
is unique up to scaling. We choose an admissible fibre metric on $E_\tau$ and
denote the analytic torsion with respect to this metric and the hyperbolic
metric by $T_X(\tau)$.  We also consider the $L^2$-torsion $T^{(2)}_X(\tau)$ 
which is defined as in \cite{Lo}.  

Assume that $\dim X=2n+1$.
Then we study the asymptotic behavior of the analytic
torsion for special sequences of representations of $G$. These representations
are defined as follows.  Fix natural
numbers $\tau_{1}\geq\tau_{2}\geq\dots\geq\tau_{n+1}$. For $m\in\mathbb{N}$ let 
$\tau(m)$ be the finite-dimensional irreducible representation of $G$ with 
highest weight 
\[
\Lambda_{\tau(m)}=(\tau_{1}+m)e_{1}+\dots+(\tau_{n+1}+m)e_{n+1}
\] 
as in 
\eqref{Darstellungen von G}. Let $\theta$ be the 
standard Cartan involution of $G$. For a given representation $\tau$ of $G$,
let $\tau_\theta:=\tau\circ\theta$. Then the representation $\tau(m)$ satisfies
$\tau(m)\neq\tau(m)_\theta$.  By the vanishing theorem of Borel/Wallach
\cite[Theorem 6.7]{Borel}  it follows that $H^*(X,E_{\tau(m)})=0$. 
Thus by \cite[Corollary 2.7]{Mu1} the Ray-Singer analytic 
torsion is independent of 
the choice of a Hermitian fibre metric in $E_{\tau(m)}$ and the Riemannian 
metric on $X$.

The purpose of this paper is to study the asymptotic behavior of 
$T_{X}(\tau(m))$ as $m\to\infty$. Let $\tau:=(\tau_1,\dots,\tau_{n+1})$ and
let $\vol(X)$ denote the hyperbolic volume of $X$. Then our main result is the following theorem.
\begin{theo}\label{theo1.1}
Let $\dim X=2n+1$. 
There exist $c>0$ and a polynomial $P_{\tau}(m)$ of degree $n(n+1)/2+1$ whose 
coefficients depend only on $n$ and $\tau$, such that
\begin{align*}
\log T_{X}(\tau(m))=\vol(X)P_{\tau}(m)+O(e^{-cm})
\end{align*}
as $m\to\infty$. 
\end{theo}
To prove this theorem we first show that the asymptotic behavior of the analytic
torsion is determined by the asymptotic behavior of the $L^2$-torsion. More
precisely we have
\begin{prop}\label{prop1.2}
Let $\dim X=2n+1$. There exists $c>0$ such that
\begin{align*}
\log T_{X}(\tau(m))=\log T^{(2)}_X(\tau(m))+O(e^{-cm})
\end{align*}
as $m\to\infty$.
\end{prop}
The $L^2$-torsion is obtained from the contribution of the identity to the
Selberg trace formula. It can be computed using the 
Plancherel formula. We have
\begin{prop}\label{prop1.3}
Let $\dim X=2n+1$. 
There exists a polynomial $P_{\tau}(m)$ of degree $n(n+1)/2+1$ whose 
coefficients depend only on $n$ and $\tau$, such that
\[
\log T^{(2)}_X(\tau(m))=\vol(X)P_\tau(m).
\]
\end{prop}
Combining Propositions \ref{prop1.2} and \ref{prop1.3}, we obtain Theorem
\ref{theo1.1}.

We note that for a unitary representation $\rho$ of $\Gamma$ one has
$T^{(2)}_X(\rho)=\dim(\rho)\cdot T^{(2)}_X$, where $T^{(2)}_X$ is the 
$L^2$-torsion with
respect to the trivial representation, which equals $C(n)\cdot\vol(X)$.
This is not true for the representations
$\tau$ which arise by restriction of representations of $G$. Indeed by
Weyl's dimension formula (see \eqref{Dimension tau}) there exists a constant 
$C>0$ such that
\begin{equation}\label{weyldim}
\dim(\tau(m))=Cm^{\frac{n(n+1)}{2}}+O(m^{\frac{n(n+1)}{2}-1}),\quad m\to\infty. 
\end{equation}
But the polynomial $P_\tau(m)$ has degree $n(n+1)/2+1$. 

The coefficient of the leading term of $P_\tau(m)$ can 
be determined explicitly. Combined with \eqref{weyldim} we obtain  
\begin{corollary}\label{asympt1}
Let $\dim X=2n+1$. 
There exists a constant $C=C(n)>0$ which depends only on $n$, such that we have
\begin{align*}
-\log T_{X}(\tau(m))=C(n)\vol(X)m\cdot\dim(\tau(m))+O(m^{\frac{n(n+1)}{2}})
\end{align*}
as $m\longrightarrow\infty$.
\end{corollary}
The constant $C(n)$ can be computed explicitly from the Plancherel polynomials.
The case $n=1$ was established  in \cite{Mu2}. Corollary \ref{asympt1}
is the extension of Theorem 1.1 of \cite{Mu2} to higher dimensions and to
rays of highest weights which exhaust the space of weights in the positive 
Weyl chamber.
 
Using the equality of analytic and 
Reidemeister  torsion \cite{Mu1}, we get corresponding statements for the
Reidemeister torsion. Especially we have
\begin{corollary}\label{reidem}
Let $\dim X=2n+1$. 
Let  $\tau_X(\tau(m))$ be the Reidemeister torsion of $(X,\tau(m))$. Then
we have 
\begin{align*}
-\log \tau_{X}(\tau(m))=C(n)\vol(X)m\cdot\dim(\tau(m))+
O(m^{\frac{n(n+1)}{2}})
\end{align*}
as $m\longrightarrow\infty$.
\end{corollary}
It follows from Corollary \ref{reidem} and \eqref{weyldim} that there is
a constant $C_1(n)<0$ such that
\begin{equation}\label{reidem2}
\lim_{m\to\infty}\frac{\log \tau_X(\tau(m))}{m^{n(n+1)/2+1}}=C_1(n)\vol(X).
\end{equation}
Now recall that the Reidemeister torsion of an acyclic representation is 
defined combinatorially in terms of a smooth triangulation of $X$. 
Thus the volume appears as the limit of a sequence of numbers which are
defined pure combinatorially. 

Again, for hyperbolic $3$-manifolds Corollary \ref{reidem} was proved in
 \cite{Mu2} and 
it has been used in \cite{Marsh} to study the growth of the torsion in the 
cohomology of arithmetic hyperbolic $3$-manifolds. In the same way, 
Corollary \ref{reidem} can be used to study the torsion in the cohomology 
of arithmetic hyperbolic manifolds of odd dimension. 

Another immediate application is the following corollary which extends 
\cite[Corollary 1.4]{Mu2} to higher dimensions.
\begin{corollary}\label{reidem1}
Let $X$ be a closed, oriented hyperbolic manifold of odd dimension. Then
the volume of $X$ is determined by the set $\{\tau_X(\tau(m))\colon m\in\N\}$
of Reidemeister torsion invariants.
\end{corollary}

Finally we have the following result about the analytic torsion in even 
dimensions. 
\begin{prop}\label{tortrivial}
Assume that $\dim X$ is even.  Then $T_X(\tau)=1$ for all irreducible
finite-dimensional  representations $\tau$ of $G$.
\end{prop}
Since $\tau|_\Gamma$ is non-unitary, this is not obvious and
may be  false if we choose other fibre metrics in $E_\tau$ (see \cite{Mu1}).

Next we explain our method to prove Theorem \ref{theo1.1}. 
The approach used in \cite{Mu2} was based on the  expression of
the analytic torsion in terms of the value at zero of the corresponding 
twisted 
Ruelle zeta function. Such a relation between the twisted Ruelle zeta function 
and the 
analytic torsion continues to hold in the higher dimensional case
(see \cite{Brocker}, \cite{Wotzke}) and could be used to prove Corollary 
\ref{asympt1} in the same way.
However, we choose a more direct approach which gives the stronger result 
of Theorem \ref{theo1.1}. Moreover we expect that this method  can be used 
to deal with groups of higher rank.

For the moment let $\tau$ be any irreducible finite-dimensional representation
of $G$ and denote by $E_\tau\to X$ the flat vector bundle associated to the
restriction of $\tau$ to $\Gamma$. By \cite[Proposition 3.1]{Mats}, $E_\tau$ 
is isomorphic to the locally homogeneous vector bundle defined by the 
restriction of $\tau$ to $K$. Using this isomorphism, $E_\tau$ can be 
equipped with a canonical Hermitian fibre metric, which is induced from an
invariant metric on the corresponding  homogeneous vector bundle   
\cite[Lemma 3.1]{Mats}.
Let $\Delta_p(\tau)$ be 
the Laplace operator on $E_\tau$-valued $p$-forms with respect to an 
admissible  metric on $E_\tau$ and the hyperbolic metric of $X$. Let
\[
K(t,\tau)=\sum_{p=0}^{d}(-1)^p p\Tr\left(e^{-t\Delta_p(\tau)}\right).
\]
Assume that $\tau|_\Gamma$ is acyclic, that is $H^*(X,E_\tau)=0$. Then the 
analytic torsion is given by
\begin{equation}\label{anator}
\log T_X(\tau):=\frac{1}{2}\frac{d}{ds}\left(\frac{1}{\Gamma(s)}
\int_0^\infty t^{s-1}K(t,\tau)\;dt\right)\bigg|_{s=0}.
\end{equation}
Now we turn to the representations $\tau(m)$, $m\in\N$, defined above. As 
pointed out above, each $\tau(m)$ is acyclic. Even more is true. 
There exists $m_0\in\N$ such that for $m\ge m_0$ we have
\begin{equation}\label{lowbound4}
\Delta_p(\tau(m))\ge \frac{m^2}{2}
\end{equation}
(see Corollary \ref{lowbound1}). Let $m\ge m_0$. Since $\tau(m)$ is acyclic, 
$T_X(\tau(m))$ is metric independent \cite{Mu1}. This means that we can 
replace $\Delta_p(\tau(m))$ by $\frac{1}{m}\Delta_p(\tau(m))$. If we split
the $t$-integral into the integral over $[0,1)$ and the integral over
$[1,\infty]$, we get
\begin{equation}\label{anator6}
\begin{split}
\log T_X(\tau(m))=&\frac{1}{2}\frac{d}{ds}\left(\frac{1}{\Gamma(s)}
\int_0^1 t^{s-1}K\left(\frac{t}{m},\tau(m)\right)\,dt\right)\bigg|_{s=0}\\
&+\frac{1}{2}\int_1^\infty t^{-1}K\left(\frac{t}{m},\tau(m)\right)\,dt.
\end{split}
\end{equation}
It follows from \eqref{lowbound4} and standard estimations of the heat
kernel that the second term on the right is $O(e^{-\frac{m}{8}})$ as $m\to\infty$.
To deal with the first term, we use a preliminary form of the Selberg trace 
formula.

Let $\widetilde\Delta_p(\tau)$ be the lift of 
$\Delta_p(\tau)$ to $\tilde X$. Then $e^{-t\widetilde \Delta_p(\tau)}$ is an
invariant integral operator. Using the kernels of the heat operators 
$e^{-t\widetilde \Delta_p(\tau)}$ we construct 
a smooth $K$-finite function $k_t^{\tau(m)}$ on $G$, which belongs to 
Harish-Chandra's Schwartz space $\mathcal{C}(G)$, such that
\[
K(t,\tau(m))=\int_{\Gamma\bs G}\sum_{\gamma\in\Gamma}k^{\tau(m)}_t(g^{-1}\gamma g)
\,d\dot g.
\]
We split the integral into the contribution of the identity which equals
\[
I(t,\tau(m))=\vol(X) k^{\tau(m)}_t(1).
\]
and the integral $H(t,\tau(m))$ of the sum of the non-trivial elements.
Using methods of \cite{Donnelly} and \cite{DL} to estimate the heat kernels,
it follows that there exist $C,c_1,c_2>0$ such that
\[
\left| H\left(\frac{t}{m},\tau(m)\right)\right|\le C e^{-c_1m}e^{-c_2/t}
\]
for all $m\ge m_0$ and $0<t\le 1$. This implies that the contribution of
$H(t/m,\tau(m))$ to the first term on the right of \eqref{anator6} is
of order $O(e^{-c_1m})$ as $m\to\infty$. So we are left with the contribution
of the identity. The first observation is that the $t$-integral over $[0,1]$
can be replaced by the integral over $[0,\infty)$. The difference is 
exponentially decaying in $m$. Furthermore, we can change variables back to 
$t$. Then the identity contribution is given by
\begin{equation}\label{ic}
\frac{1}{2}\vol(X)\frac{d}{ds}\left(\frac{1}{\Gamma(s)}\int_0^\infty t^{s-1}
k_t^{\tau(m)}(1)\,dt\right)\Bigg|_{s=0}.
\end{equation}
The $t$-integral converges absolutely for $\textup{Re}(s)>d/2$ and admits a 
meromorphic
extension to $\C$ which is regular at $s=0$. Now it is easy to see that 
\eqref{ic} equals $\log T^{(2)}_X(\tau(m))$. Putting everything together, the
proof of Proposition \ref{prop1.2} follows. To compute the $L^2$-torsion we 
apply the Plancherel formula to $k^{\tau(m)}_t(1)$ and use properties of the 
Plancherel polynomials. This leads to the proof of Proposition \ref{prop1.3}.

The paper is organized as follows. In section \ref{Notations} we fix 
notation and collect a number of facts about representation theory which
are needed for this paper. In section \ref{secBLO} we prove some basic 
estimates for the heat kernel of Bochner-Laplace operators. 
In section \ref{Torsion} we relate the analytic torsion to the Selberg
trace formula and compute the Fourier transform of the corresponding test 
function. In the final section \ref{Beweis} we prove the main results. 

\section{Preliminaries}\label{Notations}
\setcounter{equation}{0}

In this section we will establish some notation and recall some basic facts 
about representations of the involved Lie groups.

\subsection{}

For $d\in\N$, $d>1$ let  $G:=\Spin(d,1)$. 
Recall that $G$ is the universal covering group 
of $\SO_{0}(d,1)$. Let
$K:=\Spin(d)$. Then $K$ is a maximal compact subgroup of $G$. Put
$\tilde{X}:=G/K$. Let 
\[
G=NAK
\]
 be the standard Iwasawa decomposition of $G$ and let $M$ be the 
centralizer of $A$ in $K$. Then $M=\Spin(d-1)$. 
The Lie algebras of  $G,K,A,M$ and $N$ will be denoted by
$\gL,\kL,\aL,\mL$ and $\nf$, respectively. Define the 
standard Cartan involution $\theta:\gL\rightarrow \gL$ by
\begin{align*}
\theta(Y)=-Y^{t},\quad Y\in\gL. 
\end{align*}
The lift of $\theta$ to $G$ will be denoted by the same letter $\theta$. Let 
\begin{align*}
\mathfrak{g}=\mathfrak{k}\oplus\mathfrak{p}
\end{align*}
be the Cartan decomposition of $\gL$ with respect to $\theta$. Let $x_0=eK\in
\tilde X$. Then we have a canonical isomorphism
\begin{equation}\label{tspace}
T_{x_0}\tilde X\cong \pL.
\end{equation}
Define the symmetric bilinear form $\langle\cdot,\cdot\rangle$ on $\gL$ by
\begin{equation}\label{killnorm}
\langle Y_1,Y_2\rangle:=\frac{1}{2(d-1)}B(Y_1,Y_2),\quad Y_1,Y_2\in\gL.
\end{equation}
By \eqref{tspace} the restriction of $\langle\cdot,\cdot\rangle$ to $\pL$ 
defines an inner product on $T_{x_0}\tilde X$ and therefore an invariant 
metric on $\tilde X$. This metric has constant curvature $-1$. Then $\tilde X$,
equipped with this metric, is isometric to the hyperbolic space $\bH^{d}$. 

Let $\Gamma\subset G$ be a discrete, 
co-compact torsion free subgroup. Then $\Gamma$ acts properly discontinuously
on $\tilde X$ and $X=\Gamma\bs\tilde X$ is a compact,
oriented hyperbolic manifold of dimension d. Moreover any such manifold is 
of this form.

\subsection{}
Let now $d=2n+1$. Denote by $E_{i,j}$ the matrix in $\gL$ whose entry at the
i-th row 
and j-th column is equal to 1 and all of its other entries are equal to 0. Let
\begin{equation}\label{basis1}
H_i:=\begin{cases}E_{1,2}+E_{2,1}, & i=1;\\
\sqrt{-1}(E_{2i-1,2i}-E_{2i,2i-1}), & i=2,\dots n+1.
\end{cases}
\end{equation}
Then
\[
\aL=\R H_1
\]
and 
\[
\mathfrak{b}=\mathbb{R}\sqrt{-1}H_{2}+\dots+\mathbb{R}\sqrt{-1}H_{n+1}
\]
is the 
standard Cartan subalgebra of $\mathfrak{m}$. Moreover $\mathfrak{b}$ is also a 
Cartan subalgebra of $\mathfrak{k}$, and
\begin{align*}
\mathfrak{h}:=\mathfrak{a}\oplus\mathfrak{b}
\end{align*}
is a Cartan-subalgebra of $\mathfrak{g}$. Define 
$e_{i}\in\mathfrak{h}_{\mathbb{C}}^{*}$, $i=1,\dots,n+1$,  by
\begin{align*}
e_{i}(H_{j})=\delta_{i,j},\: 1\leq i,j\leq n+1.
\end{align*}
Then the sets of roots of $(\gL_\C,\hL_\C)$, $(\kL_\C,\bL_\C)$ and $(\mL_\C,
\bL_\C)$ are given by
\begin{align*}
&\Delta(\mathfrak{g}_{\C},\mathfrak{h}_{\mathbb{C}})=\{\pm e_{i}\pm e_{j},\: 1
\leq i<j\leq n+1\}\\
&\Delta(\mathfrak{k}_{\C},\mathfrak{b}_{\C})=\{\pm e_{i},\:2
\leq i<j\leq n+1\}\sqcup\{\pm e_{i}\pm e_{j},\: 2\leq i<j\leq n+1\}\\
&\Delta(\mathfrak{m}_{\C},\mathfrak{b}_{\C})=\{\pm e_{i}\pm e_{j},\: 2
\leq i<j\leq n+1\}
\end{align*}
(see \cite[Section IV,2]{Knapp}). 
We fix positive systems of roots by
\begin{align*}
&\Delta^{+}(\mathfrak{g}_{\mathbb{C}},\mathfrak{h}_{\mathbb{C}})
:=\{e_{i}+e_{j},\:i\neq j\}\sqcup\{e_{i}-e_{j},\:i<j\}\\ 
& \Delta^{+}(\mathfrak{m}_{\mathbb{C}},\mathfrak{b}_{\mathbb{C}})
:=\{e_{i}+e_{j},\:i\neq j,\:i,j\geq 2\}\sqcup\{e_{i}-e_{j},\:2\leq i<j\}.
\end{align*}
For $j=1,\dots,n+1$ let
\begin{align*}
\rho_{j}:=n+1-j.
\end{align*}
Then the half-sum of positive roots $\rho_G$ and $\rho_M$, respectively, are
given by
\begin{align}\label{Definition von rho(G)}
\rho_{G}:=\frac{1}{2}\sum_{\alpha\in\Delta^{+}(\mathfrak{g}_{\mathbb{C}},
\mathfrak{h}_\mathbb{C})}\alpha=\sum_{j=1}^{n+1}\rho_{j}e_{j}
\end{align}
and
\begin{align}\label{Definition von rho(M)}
\rho_{M}:=\frac{1}{2}\sum_{\alpha\in\Delta^{+}(\mathfrak{m}_{\mathbb{C}},
\mathfrak{b}_{\mathbb{C}})}\alpha=\sum_{j=2}^{n+1}\rho_{j}e_{j}.
\end{align}
Let $W_{G}$ be the Weyl-group of $\Delta(\mathfrak{g}_{\C},
\mathfrak{h}_{\mathbb{C}})$. 

\subsection{}

Let ${{\mathbb{Z}}\left[\frac{1}{2}\right]}^{j}$ be the set of all 
$(k_{1},\dots,k_{j})\in\mathbb{Q}^{j}$ such that either all $k_{i}$ are 
integers or all $k_{i}$ are half integers. Let $d=2n+1$. Then the
finite-dimensional irreducible representations $\tau\in\hat{G}$ of $G$ are 
parametrized by their highest weights
\begin{equation}\label{Darstellungen von G}
\begin{split}
\Lambda(\tau)=&k_{1}(\tau)e_{1}+\dots+k_{n+1}(\tau)e_{n+1},\quad
(k_{1}(\tau),\dots k_{n+1}(\tau))\in{{\mathbb{Z}}
\left[\frac{1}{2}\right]}^{n+1}\\
&k_{1}(\tau)\geq k_{2}(\tau)
\geq\dots\geq k_{n}(\tau)\geq \left|k_{n+1}(\tau)\right|.
\end{split}
\end{equation}
Furthermore the finite dimensional representations $\nu\in\hat{K}$ of $K$ are
parametrized by their highest weights 
\begin{equation}\label{Darstellungen von K}
\begin{split}
\Lambda(\nu)=&k_{2}(\nu)e_{2}+\dots+k_{n+1}(\nu)e_{n+1},\quad
(k_{2}(\nu),\dots k_{n+1}(\nu))\in{{\mathbb{Z}}
\left[\frac{1}{2}\right]}^{n}\\
&k_{2}(\nu)\geq k_{2}(\nu)
\geq\dots\geq k_{n}(\nu)\geq k_{n+1}(\nu)\geq 0.
\end{split}
\end{equation}
Finally the  finite-dimensional irreducible representations 
$\sigma\in\hat{M}$ of $M$ 
are parametrized by their highest weights
\begin{equation}\label{Darstellungen von M}
\begin{split}
\Lambda(\sigma)=&k_{2}(\sigma)e_{2}+\dots+k_{n+1}(\sigma)e_{n+1},
\quad(k_{2}(\sigma),\dots, k_{n+1}(\sigma))\in {{\mathbb{Z}}
\left[\frac{1}{2}\right]}^{n}, \\ 
&k_{2}(\sigma)\geq 
k_{3}(\sigma)\geq\dots\geq k_{n}(\sigma)\geq \left|k_{n+1}(\sigma)\right|.
\end{split}
\end{equation}
For $\tau\in\hat{G}$ let $\tau_{\theta}:=\tau\circ\theta$. Let $\Lambda(\tau)$ 
denote the highest weight of $\tau$ as in \eqref{Darstellungen von G}.
Then the highest weight $\Lambda(\tau_\theta)$ of $\tau_\theta$ is given by
\begin{equation}\label{Tau theta}
\Lambda(\tau_{\theta})=k_{1}(\tau)e_{1}+\dots+k_{n}(\tau)e_{n}-k_{n+1}(\tau)e_{
n+1}.
\end{equation}
Moreover, by the Weyl dimension formula \cite[Theorem 4.48]{Knapp} we have
\begin{equation}\label{Dimension tau}
\begin{split}
\dim(\tau)=&\prod_{\alpha\in\Delta^{+}(\mathfrak{g}_{\mathbb{C}},
\mathfrak{h}_{\mathbb{C}})}\frac{\left<\Lambda(\tau)+\rho_{G},\alpha\right>}
{\left<\rho_{G},\alpha\right>}\\
=&\prod_{i=1}^{n}\prod_{j=i+1}^{n+1}\frac{\left(k_{i}(\tau)+\rho_{i}\right)^{2}
-\left(k_{j}(\tau)+\rho_{j}\right)^{2}}{\rho_{i}^{2}-\rho_{j}^{2}}.
\end{split}
\end{equation}
Similarly, for $\sigma\in\hat{M}$ with highest weight 
$\Lambda(\sigma)\in\mathfrak{b}_{\mathbb{C}}^{*}$ as in 
\eqref{Darstellungen von M} we have
\begin{equation}\label{Dimension sigma}
\begin{split}
\dim(\sigma)=&\prod_{\alpha\in\Delta^{+}(\mathfrak{m}_{\mathbb{C}},
\mathfrak{b}_{\C})}\frac{\left<\Lambda(\sigma)+\rho_{M},\alpha\right>}
{\left<\rho_{M},\alpha\right>}\\
=&\prod_{i=2}^{n}\prod_{j=i+1}^{n+1}\frac{\left(k_{i}(\sigma)
+\rho_{i}\right)^{2}-\left(k_{j}(\sigma)+\rho_{j}\right)^{2}}
{\rho_{i}^{2}-\rho_{j}^{2}}.
\end{split}
\end{equation}

For $\tau\in\hat{G}$ and $\nu\in\hat{K}$ we will denote by
$\left[\tau:\nu\right]$
the multiplicity of $\nu$ in the restriction of $\tau$ to $K$.
These multiplicities are described in the following proposition. 
\begin{prop}\label{Einschr}
Let $\tau\in\hat{G}$ be of highest weight $\Lambda(\tau)$ as in
\eqref{Darstellungen von G}. Then $\tau$ decomposes with multiplicity one into
representations $\nu\in\hat{K}$ with highest weight $\Lambda(\nu)$
as in \eqref{Darstellungen von K} such that $k_{j-1}(\tau)\geq k_j(\nu)\geq
\left|k_j(\tau)\right|$ for every $j\in\{2,\dots,n+1\}$ and such that
all $k_j(\nu)$ are integers if all $k_j(\tau)$ are integers resp. such that
all $k_j(\nu)$ are half-integers if all $k_j(\tau)$ are half integers.
\end{prop}
\begin{proof}
\cite{Goodman}[Theorem 8.1.4]
\end{proof}
Let $M'$ be the normalizer of $A$ in $K$ and let $W(A)=M'/M$ be the 
restricted Weyl-group. It has order two and it acts on the finite-dimensional 
representations of $M$ as follows. Let $w_{0}\in W(A)$ be the non-trivial 
element and let $m_0\in M^\prime$ be a representative of $w_0$. Given 
$\sigma\in\hat M$, the representation $w_0\sigma\in \hat M$ is defined by
\[
w_0\sigma(m)=\sigma(m_0mm_0^{-1}),\quad m\in M.
\] 
Let $\Lambda(\sigma)=k_{2}(\sigma)e_{2}+\dots+k_{n+1}(\sigma)e_{n+1}$ be the 
highest weight
of $\sigma$ as in \eqref{Darstellungen von M}. Then the highest weight 
$\Lambda(w_0\sigma)$ of $w_0\sigma$ is given by
\begin{equation}\label{wsigma}
\Lambda(w_0\sigma)=k_{2}(\sigma)e_{2}+\dots+k_{n}(\sigma)e_{n}
-k_{n+1}(\sigma)e_{n+1}.
\end{equation}

Let $R(K)$ and $R(M)$ be the representation rings of $K$ and $M$. Let 
$\iota:M\longrightarrow K$ be the inclusion and let $\iota^{*}:R(K)
\longrightarrow R(M)$ be the induced map. If $R(M)^{W(A)}$ is the subring of 
$W(A)$-invariant elements of $R(M)$, then clearly $\iota^{*}$ maps $R(K)$ into 
$R(M)^{W(A)}$. 
\begin{prop}\label{Branching Rules}
The map $\iota^{*}$ is an isomorphism from $R(K)$ onto $R(M)^{W(A)}$.
\end{prop}
\begin{proof}
\cite{Bunke} , Proposition 1.1.
\end{proof}

\subsection{}

We parametrize the principal series as follows. Given $\sigma\in\hat{M}$ with
$(\sigma,V_\sigma) \in \sigma$, let $\mathcal{H}^{\sigma}$ denote the space of
measurable functions $f\colon K\to V_\sigma$ satisfying
\[
f(mk)=\sigma(m)f(k),\quad\forall k\in K,\, \forall m\in M,
\quad\textup{and}\quad
\int_K\parallel f(k)\parallel^2\,dk=\parallel f\parallel^2<\infty.
\]
Then for $\lambda\in\mathbb{C}$ and $f\in H^{\sigma}$ let
\begin{align*}
\pi_{\sigma,\lambda}(g)f(k):=e^{(i\lambda e_{1}+\rho)\log{H(kg)}}f(\kappa(kg)).
\end{align*}
Recall that the representations $\pi_{\sigma,\lambda}$ are unitary iff 
$\lambda\in\mathbb{R}$. Moreover, for $\lambda\in\mathbb{R}-\{0\}$ and 
$\sigma\in\hat{M}$ the representations $\pi_{\sigma,\lambda}$ are irreducible 
and $\pi_{\sigma,\lambda}$ and $\pi_{\sigma',\lambda'}$, $\lambda,
\lambda'\in\mathbb{C}$, are 
equivalent iff either $\sigma=\sigma'$, $\lambda=\lambda'$ or
$\sigma'=w_{0}\sigma$, 
$\lambda'=-\lambda$.
The restriction of $\pi_{\sigma,\lambda}$ to $K$ coincides with the induced 
representation ${\rm{Ind}}_{M}^{K}(\sigma)$. Hence by Frobenius 
reciprocity \cite[p.208]{Knapp} for every $\nu\in\hat{K}$ one has
\begin{align}\label{Frobeniusrez}
\left[\pi_{\sigma,\lambda}:\nu\right]=\left[\nu:\sigma\right].
\end{align}

\subsection{}

From now on we assume that $d=2n+1$. We establish some facts about 
infinitesimal characters. Let $U(\mathfrak{g}_{\mathbb{C}})$ denote the
universal enveloping algebra of 
$\mathfrak{g}_{\mathbb{C}}$ and let $Z(U(\mathfrak{g}_{\mathbb{C}}))$ be its 
center. Let $\Omega\in Z(U(\mathfrak{g}_{\mathbb{C}}))$ be the Casimir element
with 
respect to the Killing form normalized as in \eqref{killnorm}.
Let $I(\mathfrak{h}_{\mathbb{C}})$ be the Weyl-group invariant elements 
of  the symmetric algebra $S(\mathfrak{h}_{\mathbb{C}})$ of 
$\mathfrak{h}_{\mathbb{C}}$. Let 
\begin{align}\label{Definition des HC-Isomorphismus}
\gamma:Z(U(\mathfrak{g}_{\mathbb{C}}))\longrightarrow
I(\mathfrak{h}_{\mathbb{C}})
\end{align} 
be the Harish-Chandra isomorphism
\cite[Section VIII,5]{Knapp}. Every 
$\Lambda\in\mathfrak{h}_{\mathbb{C}}^{*}$ defines a homomorphism 
\begin{align*}
\chi_{\Lambda}:Z(U(\mathfrak{g}_{\mathbb{C}}))\longrightarrow\mathbb{C}
\end{align*}
by
\begin{align*}
\chi_{\Lambda}(Z):=\Lambda(\gamma(Z)).
\end{align*}
Let $\tau$ be an irreducible finite-dimensional representation of $G$ with 
highest weight 
$\Lambda(\tau)$. Its infinitesimal character will also be denoted by 
$\tau$, 
i.e. every $Z\in Z(U(\mathfrak{g}_{\mathbb{C}}))$ acts by 
$\tau(Z)\cdot{\rm{Id}}$. It follows from the definition of $\gamma$ that
\begin{align}\label{Zusammenhang HC Iso, Gewicht}
\tau(Z)=\chi_{\Lambda(\tau)+\rho_{G}}(Z)=\chi_{w(\Lambda(\tau)+\rho_{G})}
(Z);\quad 
Z\in Z(U(\mathfrak{g}_{\mathbb{C}})),\:w\in W.
\end{align}
Moreover, a standard computation gives
\begin{align}\label{HC-Iso auf Casimir}
\gamma(\Omega)=\sum_{j=1}^{n+1} H_{j}^{2}-\sum_{j=1}^{n+1}\rho_{j}^{2},
\end{align}
where the $H_{j}$ are defined by \eqref{basis1}.
Thus, if the highest weight $\Lambda(\tau)$ of $\tau$ is as in
\eqref{Darstellungen von G} 
one obtains
\begin{align}\label{tauomega}
\tau(\Omega)=\sum_{j=1}^{n+1}\left(k_j(\tau)+\rho_j\right)^2-\sum_{j=1}^{n+1}
\rho_j^2
\end{align}
Now let $\Omega_{K}$ be the Casimir
operator of $\mathfrak{k}$ with respect to the restriction of the normalized
Killing form on $\mathfrak{g}$ to $\mathfrak{k}$.
Then $\Omega_{K}$ belongs to $Z(U(\mathfrak{k}_\C))$, the center of the
universal enveloping algebra of $\mathfrak{k}_{\C}$. If $\nu\in\hat{K}$, we
will
denote the infinitesimal character of $\nu$ by $\nu$ too. If the highest weight
$\Lambda(\nu)$ of $\nu$ given by \eqref{Darstellungen von K}, an argument
analogous to the one above gives
\begin{align}\label{nuomega}
\nu(\Omega_K)=\sum_{j=2}^{n+1}\left(k_{j}(\nu)+\rho_{j}+\frac{1}{2}\right)^{2}
-\sum_{j=2}^{n+1}\left(\rho_{j}+\frac{1}{2}\right)^{2}.
\end{align}

\subsection{}

Now we come to the infinitesimal character of $\pi_{\sigma,\lambda}$.
\begin{prop}\label{Lemma Casimireigenwert der HS eins}
Let $\sigma\in\hat{M}$ with highest weight 
$\Lambda(\sigma)\in\mathfrak{b}_{\mathbb{C}}^{*}$. Then the infinitesimal 
character of $\pi_{\sigma,\lambda}$ equals $\chi_{\Lambda(\sigma)
+\rho_{M}+i\lambda e_{1}}$.
\end{prop}
\begin{proof}
\cite{Knapp}, Proposition 8.22.
\end{proof}

\begin{corollary}\label{casimirhs}
For $\sigma\in\hat{M}$ with highest weight $\Lambda(\sigma)$ given by
\eqref{Darstellungen von M}, let
\begin{align}\label{csigma}
c(\sigma):=\sum_{j=2}^{n+1}(k_{j}(\sigma)+\rho_{j})^{2}-\sum_{j=1}^{n+1}
\rho_{j}^{2}.
\end{align}
Then for the Casimir element $\Omega\in Z(\gL_\C)$ one has
\begin{align}
\pi_{\sigma,\lambda}(\Omega)=-\lambda^{2}+c(\sigma).
\end{align}
\end{corollary}
\begin{proof}
This follows form equation \eqref{HC-Iso auf Casimir} and Proposition 
\ref{Lemma Casimireigenwert der HS eins}. 
\end{proof}

\subsection{}

For $\sigma\in\hat{M}$ and $\lambda\in\mathbb{R}$ let $\mu_{\sigma}(\lambda)$ 
be the Plancherel measure associated to $\pi_{\sigma,\lambda}$. Then, since 
${\rm{rk}}(G)>{\rm{rk}}(K)$, $\mu_{\sigma}(\lambda)$ is a polynomial in 
$\lambda$ of degree $2n$. Let $\left<\cdot,\cdot\right>$ be the bilinear form
defined by 
\eqref{killnorm}. Let $\Lambda(\sigma)\in\mathfrak{b}_{\mathbb{C}}^{*}$ be 
the highest weight of $\sigma$ as in \eqref{Darstellungen von M}.
Then by theorem 13.2 in \cite{Knapp} there exists a constant $c(n)$ such that 
one has
\begin{align*}
\mu_{\sigma}(\lambda)=-c(n)\prod_{\alpha\in\Delta^{+}(\mathfrak{g}_{\mathbb{C}},
\mathfrak{h}_{\mathbb{C}})}\frac{\left<i\lambda e_{1}+\Lambda(\sigma)+\rho_{M},
\alpha\right>}{\left<\rho_{G},\alpha\right>}..
\end{align*}
The constant $c(n)$ is computed in \cite{Mi2}. By \cite{Mi2}, theorem 3.1, one 
has $c(n)>0$. 
For $z\in\mathbb{C}$  let
\begin{align}\label{plancherelmass}
P_{\sigma}(z)=-c(n)\prod_{\alpha\in\Delta^{+}(\mathfrak{g}_{\mathbb{C}},
\mathfrak{h}_{\mathbb{C}})}\frac{\left<z e_{1}+\Lambda(\sigma)+\rho_{M},
\alpha\right>}{\left<\rho_{G},\alpha\right>}.
\end{align}
One easily sees that
\begin{align}\label{P-Polynom is W-inv}
P_{\sigma}(z)=&P_{w_{0}\sigma}(z).
\end{align}

\subsection{}

Let $\tau\in\hat{G}$ and let  $\Lambda(\tau)=\tau_{1}e_{1}+\dots
+\tau_{n+1}e_{n+1}$ be its highest weight. For $w\in W$ let $l(w)$ denote its 
length with respect to 
the simple roots which define the positive roots above. Let 
\begin{align*}
W^{1}:=\{w\in W_{G}\colon w^{-1}\alpha>0\:\forall \alpha\in
\Delta(\mathfrak{m}_{\C},\mathfrak{b}_{\C})\}
\end{align*}
Let $V_{\tau}$ be the representation space of $\tau$. For $k=0,\dots,2n$ let 
$H^{k}(\overline{\mathfrak{n}},V_{\tau})$ be the cohomology of 
$\overline{\mathfrak{n}}$ with coefficients in $V_{\tau}$. Then 
$H^{k}(\overline{\mathfrak{n}},V_{\tau})$ is an $MA$ module. In our case, the 
theorem of Kostant states:
\begin{prop}\label{Prop Kostant}
In the sense of $MA$-modules one has
\begin{align*}
H^{k}(\overline{\mathfrak{n}};V_{\tau})\cong\sum_{\substack{w\in W^{1}\\ 
l(w)=k}}V_{\tau(w)},
\end{align*}
where $V_{\tau(w)}$ is the $MA$ module of highest weight $w(\Lambda(\tau)
+\rho_{G})-\rho_{G}$.
\end{prop}
\begin{proof}
for the proof see \cite[Theorem III.3]{Borel}.
\end{proof}
\begin{corollary}\label{Kostant}
As $MA$-modules we have
\begin{align*}
\bigoplus_{k=0}^{2n}(-1)^{k}\Lambda^{k}\nf^*
\otimes V_\tau=\bigoplus_{w\in W^{1}}(-1)^{l(w)}V_{\tau(w)}.
\end{align*}
\end{corollary}
\begin{proof}
Note that $\bar\nf\cong \nf^*$ as $MA$-modules. With this remark, 
the proof follows from proposition \ref{Prop Kostant} and the Poincare 
principle \cite[(7.2.3)]{Kostant}.
\end{proof}

For $w\in W^{1}$ let $\sigma_{\tau,w}$ be the representation of $M$ with 
highest weight 
\begin{equation}\label{sigmatauw}
\Lambda(\sigma_{\tau,w}):=w(\Lambda(\tau)
+\rho_{G})|_{\mathfrak{b}_{\mathbb{C}}}-\rho_{M}
\end{equation}
 and let  $\lambda_{\tau,w}\in\mathbb{C}$ such that 
\begin{equation}\label{lambdatauw}
w(\Lambda(\tau)+\rho_{G})|_{\mathfrak{a}_{\mathbb{C}}}=\lambda_{\tau,w}e_{1}.
\end{equation}
For $k=0,\dots n$ let
\begin{align}\label{lambdatau}
\lambda_{\tau,k}=\tau_{k+1}+n-k
\end{align}
and $\sigma_{\tau,k}$ be the representation of $G$ with highest weight
\begin{align}\label{sigmatau}
\Lambda_{\sigma_{\tau,k}}:=(\tau_{1}+1)e_{2}+\dots+(\tau_{k}+1)e_{k+1}
+\tau_{k+2}e_{k+2}+\dots+\tau_{n+1}e_{n+1}.
\end{align}
Then by the computations in \cite[Chapter VI.3]{Borel} one has
\begin{equation}\label{lambdadecom}
\begin{split}
\{(\lambda_{\tau,w},\sigma_{\tau,w},l(w))\colon w\in W^{1}\}
&=\{(\lambda_{\tau,k},\sigma_{\tau,k},k)\colon k=0,\dots,n\}\\
&\sqcup\{(-\lambda_{\tau,k},w_{0}\sigma_{\tau,k},2n-k)\colon k=0,\dots,n\}.
\end{split}
\end{equation}
\begin{bmrk}
Corollary \ref{Kostant} was first proved by U. Br\"ocker\cite{Brocker}  
by an elementary but tedious computation without using the theorem 
of Kostant. Also the convenient notation $\sigma_{\tau,k}$ 
and $\lambda_{\tau,k}$ is due to him.
\end{bmrk}

We will also need the following proposition.
\begin{prop}\label{Aussage uber Casimirew}
For every $w\in W^{1}$ one has
\begin{align*}
\tau(\Omega)=\lambda_{\tau,w}^{2}+c(\sigma_{\tau,w}).
\end{align*}
\end{prop}
\begin{proof}
Using \eqref{Zusammenhang HC Iso, Gewicht} and \eqref{HC-Iso auf Casimir} one 
gets
\begin{align*}
\tau(\Omega)=&\chi_{\Lambda(\tau)+\rho_{G}}(\Omega) 
=\chi_{w(\Lambda(\tau)+\rho_{G})}(\Omega)
=\chi_{\Lambda(\sigma_{\tau,w})+\rho_{M}+\lambda_{\tau}(w)e_{1}}(\Omega)
=\lambda_{\tau,w}^{2}+c(\sigma_{\tau,w}).
\end{align*}
\end{proof}

\section{Bochner Laplace operators}\label{secBLO}
\setcounter{equation}{0}

Regard $G$ as a principal $K$-fibre bundle over $\tilde{X}$. By the invariance
 of $\mathfrak{p}$ under $\Ad(K)$ the assignment
\begin{align*}
T_{g}^{{\rm{hor}}}:=\{\frac{d}{dt}\bigr|_{t=0}g\exp{tX}\colon
X\in\mathfrak{p}\} 
\end{align*}
defines a horizontal distribution on $G$. This connection is called the 
canonical connection. 
Let $\nu$ be a finite-dimensional unitary representation of $K$ on 
$(V_{\nu},\left<\cdot,\cdot\right>_{\nu})$. Let
\begin{align*}
\widetilde{E}_{\nu}:=G\times_{\nu}V_{\nu}
\end{align*}
be the associated homogeneous vector bundle over $\widetilde{X}$. Then 
$\left<\cdot,\cdot\right>_{\nu}$ induces a $G$-invariant metric 
$\widetilde{B}_{\nu}$ on $\widetilde{E}_{\nu}$. Let $\widetilde{\nabla}^{\nu}$
be the 
connection on $\widetilde{E}_{\nu}$ induced by the canonical connection. Then 
$\widetilde{\nabla}^{\nu}$ is $G$-invariant.
Let  
\begin{align*}
E_{\nu}:=\Gamma\backslash(G\times_{\nu}V_{\nu})
\end{align*}
be the associated locally homogeneous bundle over $X$. Since 
$\tilde{B}_{\nu}$ and $\widetilde{\nabla}^{\nu}$ are $G$-invariant, they push 
down to a metric $B_{\nu}$ and a connection $\nabla^{\nu}$ on $E_{\nu}$. Let
\begin{align}\label{globsect}
C^{\infty}(G,\nu):=\{f:G\rightarrow V_{\nu}\colon f\in C^\infty,\;
f(gk)=\nu(k^{-1})f(g),\,\,\forall g\in G, \,\forall k\in K\}.
\end{align}
Let
\begin{align}\label{globsect1}
C^{\infty}(\Gamma\backslash G,\nu):=\left\{f\in C^{\infty}(G,\nu)\colon 
f(\gamma g)=f(g)\:\forall g\in G, \forall \gamma\in\Gamma\right\}.
\end{align}
Let $C^{\infty}(X,\widetilde{E}_{\nu})$ resp. $C^{\infty}(X,E_{\nu})$ denote the
space of smooth sections of $\widetilde{E}_{\nu}$
resp. of $E_\nu$.
Then there are canonical isomorphism
\begin{align*}
A:C^{\infty}(X,\widetilde{E}_{\nu})\cong C^{\infty}(G,\nu), \quad
A:C^{\infty}(X,E_{\nu}) \cong C^{\infty}(\Gamma\backslash G,\nu).
\end{align*}
There are also a corresponding isometries for the spaces
$L^{2}(X,\widetilde{E}_{\nu})$ and $L^{2}(X,E_{\nu})$ of 
$L^{2}$-sections. For every $X\in\mathfrak{g}$, $g\in G$ and 
every $f\in C^{\infty}(X,\widetilde{E}_{\nu})$ one has
\begin{align*}
A(\widetilde{\nabla}^{\nu}_{L(g)_{*}X}f)(g)=\frac{d}{dt}\bigr|_{t=0}Af(g\exp{tX}
).
\end{align*}
Let
$\widetilde{\Delta}_{\nu}={\widetilde{\nabla^\nu}}^{*}{\widetilde{\nabla}}^{\nu}
$
be 
the Bochner-Laplace operator of $\widetilde{E}_{\nu}$. Since $\widetilde{X}$ is
complete, 
$\widetilde{\Delta}_{\nu}$ with domain the smooth compactly supported sections
is 
essentially self-adjoint \cite{Chernoff}. Its self-adjoint extension
will be denoted by $\widetilde{\Delta}_\nu$ too. By \cite[Proposition
1.1]{Mi1} on $C^{\infty}(X,\widetilde{E}_{\nu})$
one has
\begin{align}\label{BLO}
\widetilde{\Delta}_{\nu}=-R_\Gamma(\Omega)+\nu(\Omega_K).
\end{align} 
Let $e^{-t\tilde{\Delta}_{\nu}}$ be the corresponding heat semigroup on
$L^2(G,\nu)$, where $L^{2}(G,\nu)$ is defined analogously to 
\eqref{globsect}. 
Then the same arguments as in \cite[section1]{Cheeger} imply that there 
exists a function
\begin{align}\label{DefK}
K_t^\nu\in C^{\infty}(G\times G,\End(V_\nu)),
\end{align}
which is symmetric in the $G$-variables and for which $g'\mapsto K_t^\nu(g,g')$
belongs to $L^2(G,\End(V_\nu))$ for 
each $g\in G$ 
such that
\begin{align*}
K_t^\nu(gk,g'k')=\nu(k^{-1})K_t^\nu(g,g')\nu(k')v,\: \forall g,g'\in
G,\: \forall k\in K,\:\forall v\in V_\nu
\end{align*}
and such that
\begin{align*}
e^{-t\widetilde{\Delta}_\nu}\phi(g)=\int_{G}K_t^\nu(g,g')\phi(g')dg',\:\phi\in
L^2(G,\nu).
\end{align*}
Since $\Omega$ is $G$-invariant, $K_\nu$ is invariant under the diagonal action
of $G$.
Hence there exists a function 
\begin{align}\label{DefH}
{H}^{\nu}_{t}:G\longrightarrow {\rm{End}}(V_{\nu});\quad
{H}^{\nu}_{t}(k^{-1}gk')=\nu(k)^{-1}\circ {H}^{\nu}_{t}(g)\circ\nu(k'),
\:\forall k,k'\in K, \forall g\in G
\end{align}
such that
\begin{align}\label{KH}
K_t^\nu(g,g')=H_t^\nu(g^{-1}g'),\quad\forall g,g'\in G.
\end{align}
Thus one has
\begin{align*}
(e^{-t\tilde{\Delta}_{\nu}}\phi)(g)=\int_{G}{{H}^{\nu}_{t}(g^{-1}g')\phi(g')dg'}
,
\quad\phi\in  L^{2}(G,\nu),\quad g\in G.
\end{align*}
By the  arguments of \cite{Barbasch}, Proposition 2.4, $H^\nu_t$ belongs to all 
Harish-Chandra Schwartz spaces 
$(\mathcal{C}^{q}(G)\otimes {\rm{End}}(V_{\nu}))$, $q>0$.
Now let $\left\|H_t^\nu(g)\right\|$ be the norm of $H_t^\nu(g)$ in
$\End(V_\nu)$. 
Then by the principle of semigroup domination 
$\left\|H_t^{\nu}(g)\right\|$ is bounded by the scalar heat kernel. 
\begin{prop}\label{Kato}
Let $\widetilde{\Delta}_0$ be the Laplacian on functions on $\widetilde{X}$ and
let $H^0_t$ be 
the associated heat-kernel as above. Let $\nu\in\hat{K}$. Then for every $t\in
\left(0,\infty\right)$ and every $g\in G$ one has
\begin{align*}
\left\|H_t^{\nu}(g)\right\|\leq H^0_t(g).
\end{align*}
\end{prop}
\begin{proof}
First we remark that by \eqref{KH} and \cite[Lemma 1.1]{Cheeger} one has
$H_t^0(g)>0$ for every 
$t\in\left(0,\infty\right)$ and every $g\in G$.
Now using \eqref{KH} one can adapt the proof of Theorem 4.3 in \cite{DL} to our
situation.
\end{proof}
Now we pass to the quotient $X=\Gamma\backslash\widetilde{X}$.
Let $\Delta_\nu={\nabla^\nu}^*\nabla^\nu$ the closure of the Bochner-Laplace
operator
with domain the smooth sections of $E_\nu$. Then $\Delta_\nu$ is selfadjoint 
and on $C^\infty(\Gamma\backslash G,\nu)$ it induces the operator
$-R_\Gamma(\Omega)+\nu(\Omega_K)$.
Let $e^{-t\Delta\nu}$ be the heat semigroup of $\Delta_\nu$ on
$L^2(\Gamma\backslash
G,\nu)$. Then $e^{-t\Delta_\nu}$ is represented by the smooth kernel
\begin{align}\label{KernX}
H_\nu(t,x,x^\prime):=\sum_{\gamma\in\Gamma}H_t^\nu(g^{-1}\gamma g^\prime),
\end{align}
where $H_t^\nu$ is as above and where $x=\Gamma g$, $x^\prime=\Gamma g^\prime$
with
$g,g^\prime\in G$. The convergence of the series in \eqref{KernX} can be
established
for example using proposition \ref{Kato} and the methods from the proof of
Proposition \ref{esthyp} below.
It follows that the trace of the heat operator $e^{-t\Delta_\nu}$ is given by
\begin{align*}
\Tr\left(e^{-t\Delta_\nu}\right)=\int_{X}\tr H_\nu(t,x,x)dx,
\end{align*}
where $\tr:\End(V_\nu)\rightarrow \C$.
Thus if for $g\in G$ one lets
\begin{align}\label{hnu}
h_t^\nu(g):=\tr(H_t^\nu(g)),
\end{align}
one obtains
\begin{align}\label{TrBLO}
\Tr(e^{-t\Delta_\nu})=\int_{\Gamma\bs G}
\sum_{\gamma\in\Gamma}h_t^\nu(g^{-1}\gamma g)d\dot g.
\end{align}
Using results of Donnelly we now prove an estimate for the heat kernel $H_t^0$
of the Laplacian $\Delta_0$ acting on
functions on $\widetilde{X}$. 
\begin{prop}\label{esthyp}
There exist constants $C_0$ and
$c_0$
such that for every $t\in\left(0,1\right]$ and every $g\in G$ one has
\begin{align*}
\sum_{\substack{\gamma\in\Gamma\\ \gamma\neq
1}}H_t^0(g^{-1}\gamma g)\leq C_0 e^{-c_0/t}.
\end{align*}
\end{prop}
\begin{proof}
For $x,y\in\widetilde{X}$ let $\rho(x,y)$ denote the geodesic distance of $x,y$.
Then using \eqref{KH} it follows from \cite[Theorem 3.3]{Donnelly} that there
exists a
constant $C_1$ such that
for every $g\in G$ and every $t\in\left(0,1\right]$ one has
\begin{align}\label{Don}
H_t^0(g)\leq C_1 t^{-\frac{d}{2}}\exp{\left(-\frac{\rho^2(gK,1K)}{4t}\right)}.
\end{align}
Let $x\in\widetilde{X}$ and let $B_R(x)$ be the metric ball around $x$ of radius
$R$. Then one has
\begin{align}\label{volgr}
\vol B_R(x)\leq C_2 e^{2nR}.
\end{align}
Since $\Gamma$ is cocompact and torsion-free, there exists an $\epsilon>0$ such
that
$B_\epsilon(x)\cap \gamma B_\epsilon(x)=\emptyset$ for every 
$\gamma\in\Gamma-\{1\}$ and every $x\in\widetilde{X}$. Thus for every 
$x\in \widetilde{X}$ the union over
all
$\gamma B_\epsilon (x)$,
where $\gamma\in\Gamma$ is such that $\rho(x,\gamma x)\leq R$ 
is disjoint and is contained in $B_{R+\epsilon}(x)$.
Using \eqref{volgr} it follows that there
exists a constant $C_3$ such that for every $x\in\widetilde{X}$ one has
\begin{align*}
\#\{\gamma\in\Gamma\colon \rho(x,\gamma x)\leq R\}\leq C_3 e^{2nR}. 
\end{align*}
Hence there exists a
constant $C_4>0$ such that for every $x\in \widetilde{X}$ one has
\begin{align}\label{sum}
\sum_{\substack{\gamma\in\Gamma\\ \gamma\neq
1}}e^{-\rho^{3/2}(\gamma x,x)}\leq C_4.
\end{align}
Now let
\begin{align*}
c:=\inf\{\rho(x,\gamma x)\colon \gamma\in\Gamma-\{1\},\:x\in\widetilde{X}\}.
\end{align*}
We have $c>0$ and using \eqref{Don} and \eqref{sum} we get constants $c_0$ and
$C_0$ such that for every $g\in G$ and $0< t\le 1$  we have
\begin{align*}
\sum_{\substack{\gamma\in\Gamma\\ \gamma\neq
1}}H_t^0(g^{-1}\gamma g)\leq
C_1t^{-\frac{d}{2}}e^{-\sqrt{c}/t}\sum_{\substack{\gamma\in\Gamma\\ \gamma\neq
1}}e^{-
\rho^{3/2}(\gamma gK,gK)}\leq C_0e^{-c_0/t}.
\end{align*}
\end{proof}

\section{The analytic torsion}\label{Torsion}
\setcounter{equation}{0}

Let $\tau$ be an irreducible finite-dimensional representation of $G$ on $
V_{\tau}$. Let $E'_{\tau}$ be the flat vector bundle over $X$ associated to the 
restriction of $\tau$ to $\Gamma$. Then $E'_{\tau}$ is canonically isomorphic 
to the locally homogeneous vector bundle $E_{\tau}$ associated to $\tau|_{K}$ 
(see \cite[Proposition 3.1]{Mats}. By \cite[Lemma 3.1]{Mats}, there exists an 
inner product $\left<\cdot,\cdot\right>$ on $V_{\tau}$ such that 
\begin{enumerate}
\item $\left<\tau(Y)u,v\right>=-\left<u,\tau(Y)v\right>$ for all 
$Y\in\mathfrak{k}$, $u,v\in V_{\tau}$
\item $\left<\tau(Y)u,v\right>=\left<u,\tau(Y)v\right>$ for all 
$Y\in\mathfrak{p}$, $u,v\in V_{\tau}$.
\end{enumerate}
Such an inner product is called admissible. It is unique up to scaling. Fix an 
admissible inner product. Since $\tau|_{K}$ is unitary with respect to this 
inner product, it induces a metric on $E_{\tau}$ which we also call 
admissible. Let $\Lambda^{p}(E_{\tau})=\Lambda^pT^*(X)\otimes E_\tau$. Let
\begin{align}\label{repr4}
\nu_{p}(\tau):=\Lambda^{p}\Ad^{*}\otimes\tau:\:K\rightarrow{\rm{GL}}
(\Lambda^{p}\mathfrak{p}^{*}\otimes V_{\tau}).
\end{align}
Then there is a canonical isomorphism
\begin{align}\label{pforms}
\Lambda^{p}(E_{\tau})\cong\Gamma\backslash(G\times_{\nu_{p}(\tau)}
\Lambda^{p}\mathfrak{p}^{*}\otimes V_{\tau}).
\end{align}
of locally homogeneous vector bundles. 
Let  $\Lambda^{p}(X,E_{\tau})$ be the space the smooth $E_{\tau}$-valued 
$p$-forms on $X$. The isomorphism \eqref{pforms} induces an isomorphism
\begin{align}\label{isoschnitte}
\Lambda^{p}(X,E_{\tau})\cong C^{\infty}(\Gamma\backslash G,\nu_{p}(\tau)),
\end{align}
where the latter space is defined as in \eqref{globsect1}. A corresponding 
isomorphism also holds for the spaces of $L^{2}$-sections. 
Let $\Delta_{p}(\tau)$ be the 
Hodge-Laplacian on $\Lambda^{p}(X,E_{\tau})$ with respect to the admissible 
metric in $E_\tau$. Let $R_\Gamma$ denote the right regular representation 
of $G$ in $L^2(\Gamma\bs G)$. By \cite[(6.9)]{Mats} it follows that with
respect to the isomorphism \eqref{isoschnitte} one has
\[
\Delta_{p}(\tau)f=-R_\Gamma(\Omega)f+\tau(\Omega)\Id f, \: f\in
C^{\infty}(\Gamma\backslash G,\nu_{p}(\tau)).
\]
Let
\begin{align}\label{heattor1}
K(t,\tau):=\sum_{p=1}^{d}(-1)^{p}p\Tr(e^{-t\Delta_{p}(\tau)}).
\end{align}
and
\begin{equation}\label{constanttau}
h(\tau):=\sum_{p=1}^{d}(-1)^p p \dim H^p(X,E_\tau).
\end{equation}
Then $K(t,\tau)-h(\tau)$ decays exponentially as $t\to\infty$ and it follows 
from  \eqref{anator1} that
\begin{align}\label{anator2}
\log{T_{X}(\tau)}=\frac{1}{2}\frac{d}{ds}\left(\frac{1}{\Gamma(s)}
\int_{0}^{\infty}t^{s-1}(K(t,\tau)-h(\tau))\;dt\right)\bigg|_{s=0},
\end{align}
where the right hand side is defined near $s=0$ by analytic continuation of 
the Mellin transform. Let
$\widetilde{E}_{\nu_p(\tau)}:=G\times_{\nu_p(\tau)}\Lambda^p\mathfrak{p}
^*\otimes V_{\tau}$ and let
$\widetilde{\Delta}_p(\tau)$ be the lift of $\Delta_p(\tau)$ to
$C^\infty(\widetilde{X},\widetilde{E}_{\nu_p(\tau)})$.
Then again it follows from \cite[(6.9)]{Mats} that on $C^\infty(G,\nu_p(\tau))$
one has
\begin{align}\label{kuga}
\tilde\Delta_p(\tau)=-R_\Gamma(\Omega)+\tau(\Omega)\Id.
\end{align}
Let $e^{-t\tilde\Delta_p(\tau)}$, 
be the corresponding heat semigroup on $L^2(G,\nu_p(\tau))$. This is
a smoothing operator which commutes with the action of $G$. Therefore, it is of
the form
\begin{displaymath}
 \left( e^{-t\tilde\Delta_p(\tau)}\phi\right)(g)=\int_G  
H^{\tau,p}_t(g^{-1}g')\phi(g') \;dg',\quad 
\phi\in(L^2(G,\nu_p(\tau)),\;\;g\in G,
\end{displaymath}
where the kernel
\begin{align}\label{DefHH}
H^{\tau,p}_t\colon G\to\End(\Lambda^p\mathfrak p^*\otimes
V_\tau)
\end{align} 
belongs to $C^\infty\cap L^2$ and  satisfies the covariance property
\begin{equation}\label{covar}
H^{\tau,p}_t(k^{-1}gk')=\nu_p(\tau)(k)^{-1} H^{\tau,p}_t(g)\nu_p(\tau)(k')
\end{equation}
with respect to the representation \eqref{repr4}.
Moreover, for all $q>0$ we have 
\begin{equation}\label{schwartz1}
H^{\tau,p}_t \in (\mathcal{C}^q(G)\otimes
\End(\Lambda^p\pL^*\otimes V_\tau))^{K\times K}, 
\end{equation}
where $\mathcal{C}^q(G)$ denotes Harish-Chandra's $L^q$-Schwartz space. 
The proof is similar to the proof of Proposition 2.4 in \cite{Barbasch}. 
Now we come to the heat kernel of $\Delta_p(\tau)$. First 
the integral kernel of $e^{-t\Delta_p(\tau)}$ on $L^2(\Gamma\backslash
G,\nu_p(\tau)$ is given 
by 
\begin{align}\label{kernelx}
H^{\tau,p}(t;x,x'):=\sum_{\gamma\in\Gamma}{H}^{\tau,p}_{t}(g^{-1}\gamma g'),
\end{align}
where $x,x'\in \Gamma\backslash G$, $x=\Gamma g $, $x'=\Gamma g'$. As in section
\ref{secBLO} this
series converges absolutely and locally uniformly. 
Therefore
the trace of
the heat operator $e^{-t\Delta_p(\tau)}$ is given by
\[
\Tr\left(e^{-t\Delta_p(\tau)}\right)=\int_X\tr H^{\tau,p}(t;x,x)\;dx,
\]
where $\tr$ denotes the trace $\tr\colon \End(V_\nu)\to \C$.
Let
\begin{align}\label{Defh}
{h}^{\tau,p}_{t}(g):=\tr{H}^{\tau,p}_{t}(g).
\end{align}
Using \eqref{kernelx} we obtain
\begin{equation}\label{TrDeltap}
\Tr\left(e^{-t\Delta_p(\tau)}\right)=\int_{\Gamma\bs G}\sum_{\gamma\in\Gamma} 
{h}^{\tau,p}_{t}(g^{-1}\gamma g)\,d\dot g.
\end{equation}
Put
\begin{equation}\label{alter}
  k_t^\tau=\sum_{p=1}^d(-1)^pp\, h^{\tau,p}_t.
\end{equation}
Then it follows that
\begin{equation}\label{anator3}
K(t,\tau)=\int_{\Gamma\bs G}\sum_{\gamma\in\Gamma} 
k_{t}^\tau(g^{-1}\gamma g)\,d\dot g.
\end{equation}
Let $R_\Gamma$ be the right regular representation of $G$ on $L^2(\Gamma\bs G)$.
Then \eqref{anator3} can be written as
\begin{equation}\label{trace8}
K(t,\tau)=\Tr R_\Gamma(k^\tau_t).
\end{equation}
One can now apply the Selberg trace formula to the right hand side. 
For this purpose we need to compute  the Fourier transform of $k_t^{\tau}$ which
we do next.

To begin with let $\pi$ be an admissible unitary 
representation of $G$ on a Hilbert space $\cH_\pi$. Set
\[
\tilde \pi(H^{\tau,p}_t)=\int_G \pi(g)\otimes H^{\tau,p}_t(g)\,dg.
\]
This defines a bounded operator on 
$\cH_\pi\otimes \Lambda^p\mathfrak p^*\otimes V_\tau$. As in 
\cite[pp. 160-161]{Barbasch} it follows from 
\eqref{covar} that relative to the splitting
\[
\cH_\pi\otimes \Lambda^p\mathfrak p^*\otimes V_\tau=
\left(\cH_\pi\otimes \Lambda^p\mathfrak p^*\otimes V_\tau\right)^K\oplus
\left[\left(\cH_\pi\otimes \Lambda^p\mathfrak p^*\otimes V_\tau\right)^K
\right]^\perp,
\]
$\tilde \pi(H^{\tau,p}_t)$ has the form
\[
\tilde \pi(H^{\tau,p}_t)=\begin{pmatrix}\pi(H^{\tau,p}_t)& 0\\ 0& 0
\end{pmatrix}
\]
with $\pi(H^{\tau,p}_t)$ acting on $\left(\cH_\pi\otimes 
\Lambda^p\mathfrak p^*\otimes V_\tau\right)^K$. Using \eqref{kuga} it follows 
as in \cite[Corollary 2.2]{Barbasch} that
\begin{equation}\label{equtrace2}
\pi(H^{\tau,p}_t)=e^{t(\pi(\Omega)-\tau(\Omega))}\Id
\end{equation}
on $\left(\cH_\pi\otimes \Lambda^p\mathfrak p^*\otimes V_\tau\right)^K$.
Let $\{\xi_n\}_{n\in\N}$ and $\{e_j\}_{j=1}^m$ be orthonormal bases of $\cH_\pi$
and  $\Lambda^p\mathfrak p^*\otimes V_\tau$, respectively. Then we have
\begin{equation}\label{equtrace1}
\begin{split}
\Tr\pi(H_t^{\tau,p})&=\sum_{n=1}^\infty\sum_{j=1}^m\langle\pi(H_t^{\tau,p})
(\xi_n\otimes e_j),(\xi_n\otimes e_j)\rangle\\
&=\sum_{n=1}^\infty\sum_{j=1}^m\int_G\langle\pi(g)\xi_n,\xi_n\rangle
\langle H_t^{\tau,p}(g)e_j,e_j\rangle\,dg\\
&=\sum_{n=1}^\infty\int_G h_t^{\tau,p}(g)\langle\pi(g)\xi_n,\xi_n\rangle\;dg\\
&=\Tr\pi(h_t^{\tau,p}).
\end{split}
\end{equation}
Let $\pi\in\hat G$ and let $\Theta_\pi$ denote its character.
Then it follows
from \eqref{alter}, \eqref{equtrace2} and \eqref{equtrace1} that
\begin{equation}\label{equtrace3}
\Theta_\pi(k_t^\tau)=e^{t(\pi(\Omega)-\tau(\Omega))}\sum_{p=1}^d (-1)^p p\cdot
\dim(\cH_\pi\otimes\Lambda^p\pL^*\otimes V_\tau)^K.
\end{equation}
Now we let $\pi$ be a unitary principal series representation 
$\pi_{\sigma,\lambda}$, $\lambda\in\R$, $(\sigma,W_\sigma)\in\hat M$. 
Let $c(\sigma)$ be defined by \eqref{csigma}. 
By Frobenius reciprocity \cite[p.208]{Knapp} and Corollary \ref{casimirhs}
we get
\begin{equation}\label{charakterps}
\Theta_{\sigma,\lambda}(k_t^\tau)=e^{-t(\lambda^2-c(\sigma)+\tau(\Omega))}
\sum_{p=1}^d (-1)^p p\cdot
\dim(W_\sigma\otimes\Lambda^p\pL^*\otimes V_\tau)^M.
\end{equation}
Now observe that  as $M$-modules,  $\pL$ and $\aL\oplus\nf$ are equivalent. 
Thus we get 
\begin{equation}\label{alter2}
\sum_{p=1}^{d}(-1)^p p\,\Lambda^p\mathfrak  p^*
=\sum_{p=1}^{d}(-1)^{p}p\left(\Lambda^p\mathfrak n^*
+\Lambda^{p-1}\mathfrak n^*\right)
=\sum_{p=0}^{d-1}(-1)^{p+1}\Lambda^p\mathfrak n^*.
\end{equation}
Therefore we have
\begin{equation}\label{charac1}
\Theta_{\sigma,\lambda}(k^\tau_t)=e^{-t(\lambda^2-c(\sigma)+\tau(\Omega))}
\sum_{p=0}^{d-1}(-1)^{p+1}
\dim(W_\sigma\otimes\Lambda^p\nf^*\otimes V_\tau)^M.
\end{equation}
We now distinguish two cases. First assume that $d=2n$ is even.
Then for the principal series
representations of $G$ we obtain the following lemma.
\begin{lem}\label{vanishing}
Let $d$ be even. Then $\Theta_{\sigma,\lambda}(k_t^\tau)=0$ for all 
$\sigma\in\hat M$ and $\lambda\in\R$.
\end{lem}
\begin{proof}
As $M$-modules we have $\Lambda^p\nf^*\cong \Lambda^{d-1-p}\nf^*$. Hence, if
$d$ is even, we get
\[
\sum_{p=0}^{d-1}(-1)^p\Lambda^p\nf^*=0.
\]
Combined with \eqref{charac1} the lemma follows.
\end{proof}

Next we assume that $d=2n+1$ is odd. Then we have the following proposition.
\begin{prop}\label{chartau}
Assume that $d=2n+1$. Then one has
\begin{equation*}
\Theta_{\sigma,\lambda}(k_t^\tau):=\begin{cases}e^{-t(\lambda^2+\lambda_{\tau,k}
^2)}, & \sigma\in\{\sigma_{\tau,k},\:w_0\sigma_{\tau,k}\}\\
0, & \text{otherwise}.
\end{cases}
\end{equation*}
Here the $\lambda_{\tau,k}\in\R$ and the $\sigma_{\tau,k}\in\hat{M}$ are as in
\eqref{lambdatau} and \eqref{sigmatau}
.
\end{prop}
\begin{proof}
Let $\sigma\in\hat{M}$. For $w\in W^1$ let $\sigma_{\tau,w}$ be as in
\eqref{sigmatauw} and
let $V_{\sigma_{\tau,w}}$ be its representation space. 
Then by \eqref{charac1} and Corollary \ref{Kostant} we have
\begin{align*}
\Theta_{\sigma,\lambda}(k_t^\tau)=e^{-t(\lambda^2-c(\sigma)+\tau(\Omega))}\sum_{
w\in
W^1}(-1)^{l(w)+1}\dim(W_\sigma\otimes V_{\sigma_{\tau,w}})^{M}.
\end{align*}
Now observe that $\dim(W_\sigma\otimes V_{\sigma_{\tau,w}})^{M}=1$ if
$\check{\sigma}=\sigma_{\tau,w}$ and 
that $\dim(W_\sigma\otimes V_{\sigma_{\tau,w}})^{M}=0$ otherwise. Here
$\check{\sigma}$ 
denotes the contragredient representation of $\sigma$. By
\cite[section 3.2.5]{Goodman} if $n$ is
even we have $\check{\sigma}=\sigma$ 
and if $n$ is odd we have
$\check{\sigma}=w_0\sigma$ for every $\sigma\in\hat{M}$. Moreover by
\eqref{wsigma} and \eqref{csigma} 
we have $c(\sigma)=c(w_0\sigma)$ for every $\sigma\in\hat{M}$. Thus the
Proposition follows 
from \eqref{lambdadecom} and proposition \ref{Aussage uber Casimirew}.
\end{proof}

\section{Proof of the main results}\label{Beweis}
\setcounter{equation}{0}
First assume that $d=2n$. Let $\tau$ be any finite-dimensional irreducible
representation of $G$. Let $K(t,\tau)$ be defined by \eqref{heattor1}.
We use \eqref{trace8} and apply the Selberg trace formula to 
$\Tr R_\Gamma(k_t^\tau)$. 
Since $k_t^\tau$ is $K$-finite and belongs to $\mathcal{C}(G)$,
one easily sees that Theorem 6.7 in \cite{Wallach} applies also to $k_t^\tau$. 
Thus using Lemma \ref{vanishing} we get
\begin{equation}\label{anator4}
K(t,\tau)=\vol(X)k_t^\tau(1).
\end{equation}
Now we apply the Plancherel
theorem to express $k_t^\tau(1)$ in terms of characters. It follows from the 
definition of $k_t^\tau$ by \eqref{alter} and \eqref{schwartz1} that 
$k_t^\tau\in \mathcal{C}^q(G)$ for all $q>0$. Moreover, by \eqref{covar}, 
$k_t^\tau$ is left and right $K$-finite. Therefore $k_t^\tau$ is a $K$-finite
function in $\mathcal{C}(G)$. For such functions Harish-Chandra's 
Plancherel theorem holds \cite[Theorem 3]{HC}. For groups of real rank one which
have a compact Cartan subgroup it is
stated in \cite[Theorem 13.5]{Knapp}. By Lemma \ref{vanishing} the characters
of the principal series vanish on $k_t^\tau$. Therefore we have
\begin{equation}\label{planch5}
k_t^\tau(1)=\sum_{\pi\in\hat G_d}a(\pi)\Theta_\pi(k_t^\tau),
\end{equation}
where $\hat G_d$ denotes the discrete series and $a(\pi)\in\C$. Recall that
for a given $\nu\in \hat K$, there are only finitely many discrete series
representations $\pi$ with $[\pi|_K\colon\nu]\neq 0$. Since $k_t^\tau$ is
$K$-finite, the sum on the right hand side of \eqref{planch5} is finite. 
Now let $\pi\in\hat G_d$. Then $\Theta_\pi(k_t^\tau)$ is given by 
\eqref{equtrace3}. To apply this formula, we note that as $K$-modules we have
\[
\Lambda^p\pL^*\cong \Lambda^{d-p}\pL^*,\quad p=0,...,d.
\]
Since $d$ is even we get

\begin{equation}
\begin{split}
\sum_{p=0}^d(-1)^p p\Lambda^p\pL^*&
=-\sum_{p=0}^d(-1)^{d-p} (n-p)\Lambda^{d-p}\pL^*+n\sum_{p=0}^d(-1)^{d-p} 
\Lambda^{d-p}\pL^*\\
&=-\sum_{p=0}^d(-1)^{p} p\Lambda^{p}\pL^*+n\sum_{p=0}^d(-1)^{p} 
\Lambda^{p}\pL^*.
\end{split}
\end{equation}
Hence we have
\begin{equation}
\sum_{p=0}^d(-1)^p p\Lambda^p\pL^*=\frac{n}{2}\sum_{p=0}^d(-1)^{p}
\Lambda^{p}\pL^*.
\end{equation}
Using \eqref{equtrace3}, we get
\begin{equation}
\Theta_\pi(k_t^\tau)=\frac{n}{2}e^{t(\pi(\Omega)-\tau(\Omega))}\sum_{p=0}
^d(-1)^p 
\dim(\cH_\pi\otimes\Lambda^p\pL^*\otimes V_\tau)^K.
\end{equation}
Let $\cH_{\pi,K}$ be the subspace of $\cH_\pi$ consisting of all smooth 
$K$-finite vectors. Then
\[
(\cH_{\pi,K}\otimes\Lambda^p\pL^*\otimes V_\tau)^K=
(\cH_{\pi}\otimes\Lambda^p\pL^*\otimes V_\tau)^K. 
\]
So the $(\gL,K)$-cohomology $H^*(\gL,K;\cH_{\pi,K}\otimes V_\tau)$ is computed 
from the Lie algebra cohomology complex $([\cH_\pi\otimes\Lambda^*\pL^*\otimes
V_\tau]^K,d)$ (see \cite{Borel}). Thus by the Poincar\'e principle 
 we get
\begin{equation}\label{characterps3}
\Theta_\pi(k_t^\tau)=\frac{n}{2}e^{t(\pi(\Omega)-\tau(\Omega))}\sum_{p=0}
^d(-1)^p 
\dim H^p(\gL,K;\cH_{\pi,K}\otimes V_\tau).
\end{equation}
By \cite[II.3.1, I.5.3]{Borel} we have 
\[
H^p(\gL,K;\cH_{\pi,K}\otimes V_\tau)=\begin{cases}
[\cH_\pi\otimes\Lambda^p\pL^*\otimes V_\tau]^K,& \pi(\Omega)=\tau(\Omega);\\
0,& \pi(\Omega)\neq \tau(\Omega).
\end{cases}
\]
Together with \eqref{characterps3} this
implies that $\Theta_\pi(k_t^\tau)$ is independent of $t>0$. Hence by
\eqref{planch5} and \eqref{anator4} it follows that $K(t,\tau)$ is independent 
of $t>0$. Let $h(\tau)$ be the constant defined by \eqref{constanttau}. Since
$\ker\Delta_p(\tau)\cong H^p(X,E_\tau)$, $p=0,...,d$,  it follows that 
$\lim_{t\to\infty} K(t,\tau)=h(\tau)$. Thus we get
\[
K(t,\tau)-h(\tau)=0.
\]
Together with \eqref{anator2} this implies $T_X(\tau)=1$, which proves 
Lemma \ref{tortrivial}. 
\bigskip
\hfill$\boxempty$
 
Now assume that $d=2n+1$.
We fix natural numbers $\tau_{1},\dots,\tau_{n+1}$ with $\tau_{1}\geq\tau_{2}
\geq\dots\geq\tau_{n+1}$. Given $m\in\mathbb{N}$ let
$\tau(m)$ be the representation of $G$ with highest weight 
$(m+\tau_{1})e_{1}+\dots+(m+\tau_{n+1})e_{n+1}$. Then by \eqref{Tau theta} one 
has $\tau(m)\neq\tau(m)_{\theta}$ for all $m$. Hence by \cite[Theorem
6.7]{Borel}
we have $H^p(X,E_\tau(m))=0$ for all $p\in\{0,\dots,d\}$.
Using \eqref{anator2} we obtain
\begin{equation}\label{anator5}
\log T_X(\tau(m))=\frac{1}{2}\frac{d}{ds}\left(\frac{1}{\Gamma(s)}
\int_0^\infty t^{s-1}K(t,\tau(m))\,dt\right)\bigg|_{s=0}.
\end{equation}
Our goal is now to study the asymptotic behavior of $\log T_X(\tau(m))$ as
$m\to\infty$. First we prove an auxiliary result about the spectrum of 
$\widetilde{\Delta}_p(\tau(m))$. To this end for every $p\in\{0,\dots,d\}$ and
every
$m$ we define an endomorphism  $E_p(\tau(m))$ on $\Lambda^p\mathfrak{p}^*\otimes
V_{\tau(m)}$ by
\begin{align*}
E_p(\tau(m)):=\tau(m)(\Omega)\Id-\nu_p(\tau(m))(\Omega_K).
\end{align*}
Let $\widetilde{\Delta}_{\nu_p(\tau(m))}$ be the Bochner-Laplace operator on
$C^{\infty}(\widetilde{X},\widetilde{E}_{\nu_p(\tau(m))})$.
Then \eqref{BLO} and \eqref{kuga} imply that on $C^{\infty}(G,\nu_p(\tau(m)))$
one has
\begin{align}\label{BochHodg}
\widetilde{\Delta}_p(\tau(m))=\widetilde{\Delta}_{\nu_p(\tau(m))}+E_p(\tau(m)). 
\end{align}
The decomposition of $\nu_p(\tau(m))$ into its irreducible components induces 
a natural decomposition of  $C^\infty(G,\nu_p(\tau(m)))$ and of
$L^2(G,\nu_p(\tau(m)))$
With respect to this decomposition equation \eqref{BochHodg} can be rewritten as
\begin{align}\label{BochHodgg}
\widetilde{\Delta}_p(\tau(m))=\bigoplus_{\substack{\nu\in\hat{K}\\
\left[\nu_p(\tau(m)):\nu\right]\neq 0}}
\widetilde{\Delta}_\nu+\left(\tau(m)(\Omega)-\nu(\Omega_K)\right)\Id,
\end{align}
where the sum is a direct sum of unbounded operators. 
We first study $E_p(\tau(m))$. 
\begin{lem}\label{lowbound}
For $p=0,...,d$ and $m\in\N$ let 
\[
C_p(m):=\inf\{\tau(m)(\Omega)-\nu(\Omega_K)\colon \nu\in\hat{K}\colon
[\nu_p(\tau(m))\colon \nu]\neq 0\}.
\]
Then we have
\[
C_p(m)=m^2+O(m)
\]
as $m\to\infty$. 
\end{lem}
\begin{proof}
Let $\nu_p:=\Lambda^p\Ad^*\colon K\to \GL(\Lambda^p\pL^*)$, $p=0,...,d$. 
Recall that $\nu_p(\tau(m))=\tau(m)|_{K}\otimes\nu_p$. 
Let $\nu\in\hat{K}$ with
$\left[\nu_p(\tau(m)):\nu\right]\neq 0$. Then by \cite[Proposition
9.72]{Knapp2},
there exists a $\nu^\prime\in\hat{K}$ with
$\left[\tau(m):\nu^\prime\right]\neq 0$ of 
highest weight $\Lambda(\nu^\prime)\in\mathfrak{b}_{\C}^*$ and a
$\mu\in\mathfrak{b}_{\C}^*$
which is a weight of $\nu_p$ such that the heighest weight $\Lambda(\nu)$ of
$\nu$ is given
by $\mu+\Lambda(\nu^\prime)$. 
Now let $\nu^\prime\in\hat{K}$ be such that
$\left[\tau(m):\nu^\prime\right]\neq 0$. Let
$\Lambda(\nu^\prime)$ be the highest weight of $\nu^\prime$ as in
\eqref{Darstellungen von K}. 
Then by Proposition \ref{Einschr} for every $j=2,\dots,n+1$ we have
$\tau_{j-1}+m\geq k_j(\nu^\prime)\geq
\tau_{j}+m$. Moreover, every weight $\mu\in\mathfrak{b}_{\C}^*$ of $\nu_p$ is
given as
\begin{align*}
\mu=\pm e_{j_{1}}\pm \dots\pm e_{j_{p}},\quad 2\leq j_{1}<\dots<j_{p}\leq n+1.
\end{align*}
Thus, if $\nu\in\hat{K}$ is such that $\left[\nu_p(\tau(m)):\nu\right]\neq 0$,
the highest weight $\Lambda(\nu)$ of $\nu$ given as in \eqref{Darstellungen von
K}
satisfies
\begin{align*}
\tau_{j-1}+m+1\geq k_j(\nu)\geq \tau_j+m-1,\quad j\in\{2,\dots,n+1\}.
\end{align*}
The lemma follows from equation \eqref{tauomega} and equation
\eqref{nuomega}. 
\end{proof}
\begin{corollary}\label{lowbound1}
There exists $m_0\in\N$ such that for all $p=0,\dots,d$ and $m\ge m_0$ we
have
\[
\Delta_p(\tau(m))\ge \frac{m^2}{2}.
\]
\end{corollary}
\begin{proof}
By \eqref{BochHodg} we have
\[
\Delta_p(\tau(m))=\Delta_{\nu_p(\tau(m))}+E_p(\tau(m)).
\]
Now by definition we have $\Delta_{\nu_p(\tau(m))}\ge 0$. Hence the corollary 
follows from Lemma \ref{lowbound}.
\end{proof}
\begin{prop}\label{Katotr}
Let $h_t^{\tau(m),\:p}$ be defined by \eqref{Defh} and let $H_t^0$ be the heat
kernel of the Laplacian $\widetilde{\Delta}_0$ on $C^\infty(\widetilde{X})$.
There exist $m_0\in\N$ and $C_5>0$ such that for all $m\geq m_0$, $g\in G$, 
$t\in(0,\infty)$ and $p\in\{0,\dots,d\}$ one has
\begin{align*}
\left|h_t^{\tau(m),\:p}(g)\right|\leq
C_5\dim(\tau(m))e^{-t\frac{m^2}{2}} H_t^0(g).
\end{align*}
\end{prop}
\begin{proof}
Let $p\in\{0,\dots,n\}$. 
Let $H_t^{\nu_p(\tau(m))}$ and $H_t^{\tau(m),\:p}$ be defined by \eqref{DefH} and 
\eqref{DefHH}, respectively.
It follows from \eqref{BochHodg} and \eqref{BochHodgg} that
\begin{align*}
H_t^{\tau(m),\:p}(g)=e^{-tE_p(\tau(m))}\circ H_t^{\nu_p(\tau(m))}(g).
\end{align*}
Thus by proposition \ref{Kato} and lemma \ref{lowbound} there exists
an $m_0$ such that for $m\geq m_0$ one has
\begin{align*}
\left\|H_t^{\tau(m),\:p}(g)\right\|\leq
e^{-t\frac{m^2}{2}}H_t^0(g).
\end{align*}
Taking the trace in $\End(V_\tau(m)\otimes\Lambda^p\mathfrak{p}^*)$ for every 
$p\in\{0,\dots,d\}$, the
corollary follows. 
\end{proof}

Now we can continue with the study of the asymptotic behavior of 
$T_X(\tau(m))$. From now on we assume that $m\geq m_0$. 
Since $\tau(m)$ is acyclic, $T_X(\tau(m))$ is metric independent
\cite{Mu1}.
Especially we can rescale the metric by $\sqrt{m}$ without changing 
$T_X(\tau(m))$. Equivalently we can replace $\Delta_p(\tau(m))$ by $\frac{1}{m}
\Delta_p(\tau(m))$. Using \eqref{anator5} we get
\[
\log T_X(\tau(m))=\frac{1}{2}\frac{d}{ds}\left(\frac{1}{\Gamma(s)}
\int_0^\infty t^{s-1}K\left(\frac{t}{m},\tau(m)\right)\,dt\right)\bigg|_{s=0}.
\]
To continue, we split the $t$-integral into the integral over $[0,1]$ and
the integral over $[1,\infty)$. This leads to
\begin{equation}\label{splitint}
\begin{split}
\log T_X(\tau(m))=&\frac{1}{2}\frac{d}{ds}\left(\frac{1}{\Gamma(s)}
\int_0^1 t^{s-1}K\left(\frac{t}{m},\tau(m)\right)\,dt\right)\bigg|_{s=0}\\
&+\frac{1}{2}\int_1^\infty t^{-1}K\left(\frac{t}{m},\tau(m)\right)\,dt.
\end{split}
\end{equation}
We first consider the second term on the right. 
Using \eqref{anator3}, Proposition \ref{Katotr} and \eqref{TrBLO} we obtain
\[
\begin{split}
\left|K\left(\frac{t}{m},\tau(m)\right)\right|
&\leq C_5
e^{-\frac{m}{2}t}\dim(\tau(m))\int_{\Gamma\bs G}\sum_{\gamma\in\Gamma}
H_{t/m}^0(g^{-1}\gamma g)\,d\dot g\\
&= C_5 e^{-\frac{m}{2}t}\dim(\tau(m))\Tr(e^{-\frac{t}{m}\Delta_0}).
\end{split}
\]
Furthermore, by the heat asymptotic we have 
\[
\Tr(e^{-\frac{1}{m}\Delta_0})=C_d\vol(X)m^{d/2}+O\left(m^{(d-1)/2}\right)
\]
as $m\to\infty$. Hence there exists $C_6>0$ such that
\[
\left|K\left(\frac{t}{m},\tau(m)\right)\right|\le C_6 m^{d/2}
\dim(\tau(m))e^{-\frac{m}{2}t},\quad t\ge 1.
\]
Thus we obtain
\[
\left|\int_1^\infty t^{-1}K\left(\frac{t}{m},\tau(m)\right)\;dt\right|\le
C_6 m^{d/2}\dim(\tau(m))e^{-m/4}\int_1^\infty t^{-1}e^{-\frac{m}{4}t}\;dt.
\]
Using \eqref{weyldim}, it follows that
\begin{equation}\label{term2}
\int_1^\infty
t^{-1}K\left(\frac{t}{m},\tau(m)\right)\;dt=O\left(e^{-m/8}\right).
\end{equation}
Now we turn to the first term on the right hand side of \eqref{splitint}. We
need to estimate $K(t,\tau(m))$ for $0<t\le 1$. 
To this end we use \eqref{anator3} to decompose $K(t,\tau(m))$ into the sum
of two terms. The contribution of the identity is given by
\begin{align*}
I(t,\tau(m)):=\vol(X)k_t^{\tau(m)}(1)
\end{align*}
and 
\begin{align*}
H(t,\tau(m)):=\int_{\Gamma\bs G}\sum_{\substack{\gamma\in\Gamma\\ \gamma\neq
1}}k_t^{\tau(m)}(g^{-1}\gamma g)\,d\dot g
\end{align*}
is the hyperbolic contribution to $K(t,\tau(m))$. First we consider the 
hyperbolic contribution.  
Using Proposition \ref{Katotr} and Proposition \ref{esthyp}, it follows that
for every $m\geq m_0$ and every $t\in\left(0,1\right]$ we have
\begin{align*}
\sum_{\substack{\gamma\in\Gamma\\ \gamma\neq
1}}\left|k_{t}^{\tau(m)}(g^{-1}\gamma g)\right|
&\leq C_5 e^{-t\frac{m^2}{2}}\dim(\tau(m))
\sum_{\substack{\gamma\in\Gamma\\ \gamma\neq
1}}H_{t}^{0}(g^{-1}\gamma g)\\ &\leq C_6 \dim(\tau(m))
e^{-t\frac{m^2}{2}}C_0e^{-c_0/t}.
\end{align*}
Hence using \eqref{weyldim} we get
\begin{equation*}
\left| H\left(\frac{t}{m},\tau(m)\right)\right|\le C_7 e^{-c_1m}e^{-c_1/t},
\quad 0<t\le 1.
\end{equation*}
This implies that there is $c_2>0$ such that
\begin{equation*}
\frac{d}{ds}\left(\frac{1}{\Gamma(s)}\int_0^1 t^{s-1}H\left(\frac{t}{m},\tau(m)
\right)\,dt\right)\bigg|_{s=0}=\int_0^1 t^{-1}H\left(\frac{t}{m},\tau(m)
\right)\,dt=O\left(e^{-c_2m}\right)
\end{equation*}
as $m\to\infty$.\\ 
It remains to consider the contribution of the identity.  
Again $k_t^\tau$ is a $K$-finite
function in $\mathcal{C}(G)$ and thus by \cite[Theorem 3]{HC} Harish-Chandra's 
Plancherel theorem holds for $k_t^\tau$. For groups of real rank one which do
not
possess a compact Cartan subgroup
it is stated in \cite[Theorem 13.2]{Knapp}. Let the $\sigma_{\tau(m),k}$ and
$\lambda_{\tau(m),k}$, $k=0,\dots,n$, be defined by \eqref{lambdatau} and
\eqref{sigmatau}, respectively. Then for every $k$ we have
$\sigma_{\tau(m),k}\neq w_0\sigma_{\tau(m),k}$.  
Thus using \eqref{P-Polynom is W-inv} and Proposition \ref{chartau} we
obtain
\begin{align}\label{identcontr4}
I(t,\tau(m))=2\vol(X)\sum_{k=0}^n (-1)^{k+1}
e^{-t\lambda_{\tau(m),k}^2}\int_{\R}e^{-t\lambda^2}P_{\sigma_{
\tau(m),k}}(i\lambda)d\lambda.
\end{align}
Here the $P_{\sigma_{\tau(m),k}}$ are the polynomials defined in
\eqref{plancherelmass}.
The polynomials are given explicitly as follows.
\begin{lem}\label{Erstes lemma}
The Plancherel polynomial $P_{\sigma_{\tau(m),k}}(t)$ is given by
\begin{align*}
P_{\sigma_{\tau(m),k}}(t)=-c(n)(-1)^{k}\dim(\tau(m))
\prod_{\substack{j=0\\ j\neq k}}^{n}\frac{t^{2}
-\lambda_{\tau(m),j}^{2}}{\lambda_{\tau(m),k}^{2}-\lambda_{\tau(m),j}^{2}},
\end{align*}
where $c(n)$ is the 
constant occurring in the description of the Plancherel 
polynomial  by \eqref{plancherelmass}. 
\end{lem}
\begin{proof}
This is proved in Br\"ocker's thesis \cite[p. 60]{Brocker}. For convenience, 
we recall the proof. To safe notation, put
\[
\lambda_{i}:=\lambda_{\tau(m),i},\quad i=0,\dots,n.
\]
By \eqref{sigmatau} we have
\begin{equation*}
\begin{split}
\Lambda(\sigma_{\tau(m),k}) 
+\rho_{M}=&\sum_{i=2}^{k+1}(\tau_{i-1}+m+2+n-i)e_{i}+
\sum_{i=k+2}^{n+1}(\tau_{i}+m+n+1-i)e_{i}\\
&=\sum_{i=2}^{k+1}\lambda_{i-2}e_{i}+\sum_{i=k+2}^{n+1}\lambda_{i-1}e_{i}.
\end{split}
\end{equation*}
Hence by \eqref{plancherelmass} and \eqref{Dimension tau} we 
have
\begin{equation}\label{explicitpl}
\begin{split}
P_{\sigma_{\tau(m),k}}(t)=&-c(n)\prod_{j=1}^{n}\prod_{q=j+1}^{n+1}
\frac{\left<te_{1}+\sum_{i=2}^{k+1}\lambda_{i-2}e_{i}+\sum_{i=k+2}^{n+1}
\lambda_{i-1}e_{i},e_{j}+e_{q}\right>}
{\left<\sum_{l=1}^{n+1}\rho_{l}e_{l},e_{j}+e_{q}\right>}\\
&\hskip37pt\cdot\prod_{j=1}^{n}\prod_{q=j+1}^{n+1}\frac{\left<te_{1}
+\sum_{i=2}^{k+1}\lambda_{i-2}e_{i}+\sum_{i=k+2}^{n+1}
\lambda_{i-1}e_{i},e_{j}-e_{q}\right>}
{\left<\sum_{l=1}^{n+1}\rho_{l}e_{l},e_{j}-e_{q}\right>}\\
=&-c(n)\prod_{\substack{0\leq i\leq n\\ i\neq k}}\left(t^{2}
-\lambda_{i}^{2}\right)\prod_{\substack{0\leq j<i\leq n\\ i,j\neq k}}
\left(\lambda_{j}^{2}-\lambda_{i}^{2}\right)
\prod_{1\leq j<i\leq n+1}(\rho_{j}^{2}-\rho_{i}^{2})^{-1}\\
=&-c(n)(-1)^{k}\prod_{0\leq j<i\leq n}\frac{\lambda_{j}^{2}
-\lambda_{i}^{2}}{\rho_{j+1}^{2}-\rho_{i+1}^{2}}
\prod_{\substack{j=0\\ j\neq k}}^{n}\frac{t^{2}-\lambda_{j}^{2}}
{\lambda_{k}^{2}-\lambda_{j}^{2}}\\
=&-c(n)(-1)^{k}\dim(\tau(m))\prod_{\substack{j=0\\ j\neq k}}^{n}
\frac{t^{2}-\lambda_{j}^{2}}{\lambda_{k}^{2}-\lambda_{j}^{2}}.
\end{split}
\end{equation}
\end{proof}
Now recall that by \eqref{lambdatau} we have 
\[
\lambda_{\tau(m),i}=m+\tau_{i+1}+n-i.
\]
Since $\tau_i\ge \tau_j$ for $i<j$, it follows that
\[
|\lambda_{\tau(m),i}^2-\lambda_{\tau(m),j}^2|\ge 1,\quad \forall i\neq j.
\]
Using Lemma \ref{Erstes lemma}, it follows that
\[
P_{\sigma_{\tau(m),k}}(t)=\sum_{i=0}^n a_{k,i}(m)t^{2i}
\]
and there exists $C>0$ such that
\[
|a_{k,i}(m)|\le C m^{2n+n(n+1)/2}
\]
for all $k,i=0,\dots,n$ and $m\in\N$. 
Furthermore $\lambda_{\tau(m),i}\ge m$ for $i=0,...,n$. Together with Lemma 
\ref{Erstes lemma} it follows that there exist $C,c>0$ such that
\[
\left|I\left(\frac{t}{m},\tau(m)\right)\right|\le C e^{-cm}e^{-ct},\quad t\ge 1.
\]
Hence we get
\[
\int_1^\infty t^{-1}I\left(\frac{t}{m},\tau(m)\right)\,dt=O\left(e^{-cm}\right).
\]
This implies that we can replace the integral over $[0,1]$ by the integral over
$[0,\infty)$. We need the following auxiliary proposition.
\begin{prop}\label{identcontr}
Let $c>0$ and $\sigma\in\hat{M}$. For 
${\rm{Re}}(s)\ge n+1$ let 
\begin{align*}
E(s):=\int_0^\infty t^{s-1}e^{-t c^{2}}\left(
\int_{\R}e^{-t\lambda^2}P_{\sigma}(i\lambda)\;d\lambda\right) dt.
\end{align*}
Then $E(s)$ has a meromorphic continuation to $\mathbb{C}$. Moreover $E(s)$
is regular at zero and
\begin{align*}
E(0)=-2\pi\int_{0}^{c}P_{\sigma}(t)dt.
\end{align*}
\end{prop}
\begin{proof}
By \eqref{plancherelmass} every $P_{\sigma}(i\lambda)$ is an 
even polynomial in $\lambda$. The proposition is obtained using integration by
parts (see \cite{Fried} Lemma 2 and Lemma 3). 
\end{proof}
Changing variables by $t\mapsto t\cdot m$, it follows from the proposition that
\[
\frac{d}{ds}\left(\frac{1}{\Gamma(s)}\int_0^\infty t^{s-1} I\left(\frac{t}{m},
\tau(m)\right)\;dt\right)\bigg|_{s=0}
=\frac{d}{ds}\left(\frac{1}{\Gamma(s)}\int_0^\infty t^{s-1} I\left(t,
\tau(m)\right)\;dt\right)\bigg|_{s=0}.
\]
By Proposition \ref{identcontr} the Mellin transform
\[
\int_0^\infty t^{s-1}I(t,\tau(m))\,dt
\]
of $I(t,\tau(m))$ is a meromorphic function of  $s\in\C$, which is regular
at $s=0$. Let $\mathcal{M}I(\tau(m))$ denote its value at $s=0$. By
\eqref{anator5}, $\frac{1}{2}\mathcal{M}I(\tau(m))$ is the contribution of the 
identity to $\log T_X(\tau(m))$. Combining our estimates, we have proved
\begin{equation}\label{tor7}
\log T_X(\tau(m))=\frac{1}{2}\mathcal{M}I(\tau(m))+O\left(e^{-cm}\right)
\end{equation}
as $m\to\infty$.\\
Next we will identify $\frac{1}{2}\mathcal{M}I(\tau(m))$ with the $L^2$-torsion.
Recall its definition \cite{Lo}. For $p=0,\dots,d$ let 
$\Tr_\Gamma(e^{-t\tilde\Delta_p(\tau(m))})$ denote the $\Gamma$-trace of the heat
operator $e^{-t\tilde\Delta_p(\tau(m))}$ on $\tilde X$ (see \cite{Lo}). Recall that 
$e^{-t\widetilde\Delta_p(\tau(m))}$ is a convolution operator whose kernel
is given by the function
\[
H^{\tau(m),p}_t\colon G\to \End(\Lambda^p\pL^*\otimes V_{\tau(m)})
\]
which satisfies \eqref{covar}. Let $h^{\tau(m),p}_t=\tr H^{\tau(m),p}_t$. Then it 
follows that
\begin{equation}\label{gammatr}
\Tr_{\Gamma}\left(e^{-t\widetilde{\Delta}_p(\tau(m))}\right)=\vol(X)
h_t^{\tau(m),p}(1). 
\end{equation}
Let $k^{\tau(m)}_t$ be defined by \eqref{alter}. Then it follows that
\begin{equation}\label{l2-tor1}
\sum_{p=1}^d (-1)^p p\Tr_\Gamma\left(e^{-t\widetilde\Delta_p(\tau(m))}\right)=
\vol(X)k^{\tau(m)}_t(1)=I(t,\tau(m)).
\end{equation}
Using Proposition \ref{Katotr} it follows that for $m\ge m_0$ we have
\[
\Tr_{\Gamma}(e^{-t\widetilde{\Delta}_p(\tau(m))})=O\left(e^{-t\frac{m^2}{2}}\right)
\]
as $t\rightarrow \infty$. Furthermore, applying the Plancherel theorem 
to $h^{\tau(m),p}_t(1)$ and using \eqref{gammatr}, it follows that as
$t\to 0$, there is an asymptotic expansion
\[
\Tr_{\Gamma}\left(e^{-t\widetilde{\Delta}_p(\tau(m))}\right)\sim
\sum_{j\ge 0}a_j t^{-d/2+j}.
\]
This implies that
\[
\int_0^\infty t^{s-1}\sum_{p=1}^d(-1)^p p \Tr_\Gamma(e^{-t\widetilde\Delta_p(\tau(m))})
\,dt
\]
converges absolutely for $\textup{Re}(s)>d/2$ and admits a meromorphic
extension to $\C$ which is holomorphic at $s=0$.
Hence for $m\ge m_0$ the $L^2$-torsion is defined by
\[
\log T^{(2)}_X(\tau(m))=\frac{1}{2}\frac{d}{ds}\left(\frac{1}{\Gamma(s)}
\int_0^\infty t^{s-1}\sum_{p=1}^d(-1)^p p \Tr_\Gamma(e^{-t\widetilde\Delta_p(\tau(m))})
\,dt\right)\Bigg|_{s=0}.
\]
Using \eqref{l2-tor1} it follows that
\begin{equation}\label{l2-tor2}
\log T^{(2)}_X(\tau(m))=\frac{1}{2}\mathcal{M}I(\tau(m)).
\end{equation}
Combined with \eqref{tor7}, we obtain the proof of Proposition \ref{prop1.2}.

To compute the $L^2$-torsion, we observe that
by \eqref{identcontr4} and Proposition \ref{identcontr} we have
\begin{equation}\label{l2-tor3}
\mathcal{M}I(\tau(m))
=4\pi\vol(X)\sum_{k=0}^n(-1)^k\int_0^{\lambda_{\tau(m),k}}
P_{\sigma_{\tau(m),k}}(\lambda)\,d\lambda.
\end{equation}
Using the explicit form of the Plancherel polynomial as in the first
equality of \eqref{explicitpl} together with $\lambda_i=\tau_{i+1}+m+n-i$,
it follows that 
\begin{equation}\label{polyn}
P_\tau(m):=2\pi \sum_{k=0}^n(-1)^k\int_0^{\lambda_{\tau(m),k}}
P_{\sigma_{\tau(m),k}}(\lambda)\,d\lambda
\end{equation}
is a polynomial in $m$ of degree $n(n+1)/2+1$ whose coefficients 
depend only on $n$ and $\tau$. Moreover by \eqref{l2-tor2} we get
\[
\log T^{(2)}_X(\tau(m))=\vol(X)P_\tau(m), 
\]
which proves Proposition \ref{prop1.3}. 

It remains to determine the leading coefficient of $P_\tau(m)$. To this end we
 need some additional facts about the Plancherel polynomials.
\begin{lem}\label{Zweites lemma}
For every sequence $s_{0},\dots,s_{n}$, $s_{i}\neq s_{j}$ for $i\neq j$, one
has 
\begin{align*}
\sum_{k=0}^{n}\prod_{\substack{j=0\\ j\neq k}}^{n}\frac{t-s_{j}}{s_{k}-s_{j}}=1.
\end{align*}
\end{lem}
\begin{proof}
The expression is a polynomial in $t$ of order $n$ and is equal to 1 at the 
$n+1$ points $s_{0},\dots,s_{n}$. 
\end{proof}

\begin{corollary}\label{planchasymp}
One has
\begin{align*}
\sum_{k=0}^{n}(-1)^{k}P_{\sigma_{\tau(m),k}}(t)=-c(n){\rm{dim}}(\tau(m)).
\end{align*}
\end{corollary}
\begin{proof}
This follows from lemma \ref{Erstes lemma} and lemma \ref{Zweites lemma}.
\end{proof}
Now we are ready to determine the leading term. 
By \eqref{lambdatau} we have 
\[
\lambda_{\tau(m),0}>
\lambda_{\tau(m),1}>\dots>\lambda_{\tau(m),n}.
\]
Using \ref{Erstes lemma} and Corollary \ref{planchasymp}, we get
\begin{equation}\label{intplanch}
\begin{split}
\sum_{k=0}^{n}(-1)^{k}\int_{0}^{\lambda_{\tau(m),k}}P_{\sigma_{\tau(m),k}}(t)dt
&=\int_0^{\lambda_{\tau(m),n}}\sum_{k=0}^{n}(-1)^{k}
P_{\sigma_{\tau(m),k}}(t)dt\\
&\quad -c(n)\dim(\tau(m))\sum_{k=0}^{n-1}
\int_{\lambda_{\tau(m),n}}^{\lambda_{\tau(m),k}}\prod_{\substack{j=0\\ 
j\neq k}}^{n}\frac{t^{2}-\lambda_{\tau(m),j}^{2}}{\lambda_{\tau(m),k}^{2}
-\lambda_{\tau(m),j}^{2}}dt\\
&=-c(n)\lambda_{\tau(m),n}\dim\tau(m)\\
&\quad-c(n)\dim\tau(m)\sum_{k=0}^{n-1}
\int_{\lambda_{\tau(m),n}}^{\lambda_{\tau(m),k}}\prod_{\substack{j=0\\ 
j\neq k}}^{n}\frac{t^{2}-\lambda_{\tau(m),j}^{2}}
{\lambda_{\tau(m),k}^{2}-\lambda_{\tau(m),j}^{2}}\,dt.
\end{split}
\end{equation}
Now recall that by \eqref{lambdatau} we have
\begin{equation}\label{lambdatau1}
\lambda_{\tau(m),k}=\tau_{k+1}+m+n-k.
\end{equation}
Using \eqref{weyldim} it follows that the first term on 
the right hand side of \eqref{intplanch} equals 
\[
-c(n)m\dim\tau(m)+O(m^{\frac{n(n+1)}{2}}).
\]
Furthermore we have
\begin{equation}
\begin{split}
\int_{\lambda_{\tau(m),n}}^{\lambda_{\tau(m),k}}\prod_{\substack{j=0\\ 
j\neq k}}^{n}\frac{t^{2}-\lambda_{\tau(m),j}^{2}}
{\lambda_{\tau(m),k}^{2}-\lambda_{\tau(m),j}^{2}}\,dt
=\int_{\tau_{n+1}}^{\tau_{k+1}+n-k}\prod_{\substack{j=0\\j\neq k}}^{n}
\frac{(t+m)^{2}-\lambda_{\tau(m),j}^{2}}
{\lambda_{\tau(m),k}^{2}-\lambda_{\tau(m),j}^{2}}\,dt.
\end{split}
\end{equation}
Using \eqref{lambdatau1}, a direct computation shows that the integrand on 
the right hand side is bounded as $m\to\infty$. Hence we get
\begin{align}
\sum_{k=0}^{n-1}\int_{\lambda_{\tau(m),n}}^{\lambda_{\tau(m),k}}
\prod_{\substack{j=0\\ j\neq k}}^{n}
\frac{t^{2}-\lambda_{\tau(m),j}^{2}}{\lambda_{\tau(m),k}^{2}
-\lambda_{\tau(m),j}^{2}}=O(1),\: \text{as $m\rightarrow \infty$}. 
\end{align}
Using \eqref{weyldim}, we can estimate the second term on the right hand side 
of \eqref{intplanch} by $O(m^{\frac{n(n+1)}{2}})$. Together with \eqref{polyn} 
 we get
\begin{align*}
P_\tau(m)=2\pi \sum_{k=0}^{n}(-1)^{k}\int_{0}^{\lambda_{\tau(m),k}}P_{\sigma_{\tau(m),k}}(t)dt
=-2\pi c(n)m\dim(\tau(m))+O(m^{\frac{n(n+1)}{2}}),\quad
\end{align*}
as $m\rightarrow \infty$.

Now the proof of Corollary \ref{asympt1} follows from \eqref{tor7}  with 
\begin{align}\label{Explizite Form der Konstanten}
C(n):=2\pi c(n).
\end{align}
\begin{bmrk}
We have assumed that $\tau_{1},\dots,\tau_{n+1}$ are natural numbers and that 
$m\in\mathbb{N}$. Clearly, we can also assume that $\tau_{1},\dots,\tau_{n+1}$ 
and $m$ are in $\frac{1}{2}\mathbb{N}$. Then the obvious modifications give 
Theorem \ref{theo1.1} and Corollary \ref{asympt1} also in this case. 
\end{bmrk}

\begin{bmrk}\label{Letzte Bemerkung}
At the end of this section we  compute the Polynomial $P_{\tau}(m)$ in 
the three-dimensional case explicitly. In this case the group $G$ can be 
realized as ${\rm{SL}}_{2}(\mathbb{C})$, $K$ can be realized as ${\rm{SU}}(2)$ 
and $M$ can be realized as ${\rm{SO}}(2)$. If $c(n)$ is the constant in 
\eqref{plancherelmass}, it follows from \cite{Knapp}, 
Theorem 11.2, and a minor correction that
\begin{align*}
c(n)=\frac{1}{4\pi^{2}}.
\end{align*}
For $l\in\mathbb{N}$ write $\sigma_{l}$ for the representation of $M$ with 
highest weight $le_{2}$ as in \eqref{Darstellungen von M}. Then by 
\eqref{plancherelmass} one has
\begin{align*}
P_{\sigma_{l}}(z)=-\frac{1}{4\pi^{2}}(z^{2}-l^{2}). 
\end{align*}
Let $\tau_{1}\in\mathbb{N}$. Then for $\tau(m)\in\hat{G}$ with highest 
weight $(m+\tau_{1})e_{1}+me_{2}$ as in \eqref{Darstellungen von G} it follows 
that
\begin{align*}
&\sum_{k=0}^{1}(-1)^{k}\int_{0}^{\lambda_{\tau(m),k}}P_{\sigma_{\tau(m),k}}
(t)dt\\
=&-\frac{1}{4\pi^{2}}\left(\int_{0}^{\tau_{1}+m+1}(t^{2}-m^{2})dt
-\int_{0}^{m}(t^{2}-(m+\tau_{1}+1)^{2})dt\right)\\
=&-\frac{1}{4\pi^{2}}\left(2m^{2}(\tau_{1}+1)+2m(\tau_{1}+1)^{2}
+\frac{1}{3}(\tau_{1}+1)^{3}\right).
\end{align*}
Together with \eqref{Explizite Form der Konstanten} this gives
\begin{align}\label{Letzte Gleichung}
P_{\tau}(m)=-\frac{{\rm{vol}}(X)}{2\pi}\left(2m^{2}(\tau_{1}+1)+
2m(\tau_{1}+1)^{2}+\frac{1}{3}(\tau_{1}+1)^{3}\right).
\end{align}
Now if $\tau_{1}=0$, in the notation of \cite{Mu2} the representation 
$\tau(m)$ corresponds to the representation $\tau_{2m}$. Hence 
\eqref{Letzte Gleichung} is consistent with Theorem 1.1 in \cite{Mu2}. 
\end{bmrk}

\end{document}